\documentclass[11pt,reqno]{amsart}

\usepackage[colorlinks=true, pdfstartview=FitV, linkcolor=blue,citecolor=red, urlcolor=blue]{hyperref}

\usepackage{amsmath,amsthm,amssymb}
\usepackage{color}
\usepackage{hyperref}
\usepackage{url}
\newcommand{\bburl}[1]{\textcolor{blue}{\url{#1}}}

\usepackage{soul}

\newtheorem{thm}{Theorem}[section]
\newtheorem{cor}[thm]{Corollary}
\newtheorem{lem}[thm]{Lemma}
\newtheorem{prop}[thm]{Proposition}
\newtheorem{conj}[thm]{Conjecture}

\theoremstyle{definition}

\theoremstyle{definition}
\newtheorem{defi}[thm]{Definition}

\newtheorem{rem}[thm]{Remark}
\theoremstyle{remark}

\newcommand\be{\begin{equation}}
\newcommand\ee{\end{equation}}
\newcommand\bee{\begin{equation*}}
\newcommand\eee{\end{equation*}}
\newcommand\ben{\begin{enumerate}}
\newcommand\een{\end{enumerate}}

\newcommand{\R}{\ensuremath{\mathbb{R}}}
\newcommand{\C}{\ensuremath{\mathbb{C}}}
\newcommand{\Z}{\ensuremath{\mathbb{Z}}}
\newcommand{\Q}{\mathbb{Q}}

\numberwithin{equation}{section}

\DeclareMathOperator{\sgn}{sgn}

\DeclareMathOperator{\re}{Re}
\DeclareMathOperator{\im}{Im}

\DeclareMathOperator{\Res}{Res}

\newcommand{\ovA}[1]{\overline{#1\raisebox{3mm}{}}}
\newcommand{\ovB}[1]{\overline{#1\raisebox{4mm}{}}}

\usepackage{empheq}

\usepackage{mathdots}
\usepackage{bm}

\usepackage{mathabx,epsfig}

\usepackage{pst-node}
\usepackage{tikz-cd}
\usepackage{verbatim}
\usepackage{mathdots}

\usepackage{scalerel,stackengine}
\stackMath

\newcommand\reallywidehat[1]{%
	\savestack{\tmpbox}{\stretchto{%
			\scaleto{%
				\scalerel*[\widthof{\ensuremath{#1}}]{\kern.1pt\mathchar"0362\kern.1pt}%
				{\rule{0ex}{\textheight}}
			}{\textheight}%
		}{2.4ex}}%
	\stackon[-6.9pt]{#1}{\tmpbox}%
}
\usepackage{listings}
\title[Spectral Moment Formulae for $GL(3)\times GL(2)$ $L$-functions III]{Spectral Moment Formulae for $GL(3)\times GL(2)$ $L$-functions III: The Twisted Case}

\author[C.-H. Kwan]{Chung-Hang Kwan}
\email{\textcolor{blue}{\href{mailto: ucahckw@ucl.ac.uk}
{ucahckw@ucl.ac.uk}}}
\address{University College London}



\date{\today}

\begin{document}

\maketitle

\begin{abstract}
	This is a sequel to our previous articles \cite{Kw23, Kw23a+}. In this work, we  apply  recent techniques  that fall under the banner of  `Period Reciprocity' to study  moments of  $GL(3)\times GL(2)$ $L$-functions in the non-archimedean aspects,  with a view towards   the `Twisted Moment Conjectures' formulated by  CFKRS. 

\end{abstract}


\section{Introduction}\label{intr}

\subsection{Reciprocity}\label{reciprsec} This article is motivated by the important works of Young \cite{Y11}  and Blomer-Humphries-Khan-Milinovic  \cite{BHKM20} regarding the spectral theory for the fourth moment of the Dirichlet $L$-functions in the prime conductor aspect.  Their works featured a very curious  \textit{reciprocity relation}  of the form
\begin{align}\label{basicmototwis}
 \mathcal{M}_{4}(p) \ := \ \sum_{\substack{\chi \ (\bmod\, p) \\ \chi \neq \chi_{0}}}  \ \int_{|t| \ll 1} \ \left|L\left(\frac{1}{2}+it, \chi\right)\right|^4   \ dt
	\   = \  	p^{1/2} \ \sum_{\substack{f: \text{ level $1$ } \\ \text{cusp forms}}} \   L\left( \frac{1}{2}, f\right)^3 \lambda_{f}(p)    \ +   \  (***),
\end{align}
where $p$ is a prime number.  

The formula above can be recognized  as a \textit{non-archimedean} generalization of the  celebrated \textit{spectral reciprocity} formula due to Motohashi  \cite{Mo93, Mo97}:
\begin{align}\label{basicmoto}
	\int\limits_{-\infty}^{\infty}  \ \left|\zeta\left(\frac{1}{2}+it\right)\right|^4 w(t)   \ dt 
	\  \ = \ \  	\sum_{\substack{f: \text{ level $1$ } \\ \text{cusp forms}}} \  L\left( \frac{1}{2}, f\right)^3 \widecheck{w}(t_{f})     \  \ + \  \  (***). 
\end{align}
 Indeed,  (\ref{basicmototwis}) incorporates  an additional  duality between  the  $GL(1)$ twists of  Dirichlet characters and the $GL(2)$   twists of  Hecke eigenvalues when compared to  (\ref{basicmoto}). Furthermore, the reciprocity  (\ref{basicmototwis}) succinctly captures   the essential  inputs that  lead to  a sharp asymptotic formula for $\mathcal{M}_{4}(p)$ ---  it suffices to apply the spectral large sieve and the best  approximation towards the Ramanujan-Petersson conjecture for $GL(2)$ cuspidal Hecke eigenforms to the spectral side of  (\ref{basicmototwis}). 

 The arguments culminating in (\ref{basicmototwis}), as presented by  \cite{Y11} and \cite{BHKM20}, have rendered the elegant-looking non-archimedean reciprocity  even more surprising.    On the one hand, Blomer et. al. \cite{BHKM20}  devised a strategic compositum of spectral and harmonic summation formulae   which can be summarized as  `\textit{Poisson-(Arithmetic) Reciprocity-Poisson-Voronoi-Kuznetsov}'.  One arrived at the reciprocity formula (\ref{basicmototwis}) upon carefully tracing through 
 the delicate   transformations of exponential/character sums  in each step.   On the other hand,  Young \cite{Y11} applied  the spectral method in conjunction with the delta method to a \textit{generalized binary additive problem}.  His strategy was favourable in establishing  the conjecture of \cite{CFKRS05} for the moment $\mathcal{M}_{4}(p)$ and turned out to be useful in many other instances, see \cite{HY10,  BBLR20, Wu22, Za19, Be19, Ng21, HN22}.  An entirely different sequence of transformations led to the same formula (\ref{basicmototwis}) once again.  Regardless,  both \cite{Y11} and \cite{BHKM20} greatly enriched the Kuznetsov-based approach introduced by  \cite{Mo93}.

The first goal of this article is to understand the   source  of the \textit{non-archimedean} reciprocity  discovered by \cite{Y11, BHKM20} using  \textit{period integrals}. In our previous work \cite{Kw23}, we  focused on the archimedean aspects of our period identity   and a number of technical features of this article were thus not present in \cite{Kw23}.  For the sake of elucidating the key ideas and the other applications  (our second goal, see Section \ref{twisintro}),  we establish a higher-rank analogue of the reciprocity (\ref{basicmototwis}) involving  $GL(3)\times GL(2)$ Rankin-Selberg $L$-functions.  In this case, it is more natural to describe the result  in the spectral direction. Nevertheless, it is important to note that the reciprocity of our interest stays the same for both the spectral and the arithmetic directions.


\subsection{Main Results}

In this article, we work with the same admissible class of test functions   $\mathcal{C}_{\eta}$ ($\eta>40$) as in \cite{Kw23}.    The class  $\mathcal{C}_{\eta}$ consists of holomorphic functions $H$ defined on the vertical strip  $|\re \mu|< 2\eta$ \ satisfying  	$H(\mu) =   H(-\mu)$  and	$H(\mu)  \ll  e^{-2\pi|\mu|}$ on such a strip.

Let  $\Phi$ be a fixed Hecke-normalized Maass cusp form of $SL_{3}(\Z)$ with Langlands parameters $(\alpha_{1}, \alpha_{2}, \alpha_{3}) \in (i\R)^3$,  Hecke eigenvalues $\lambda_{\Phi}(m)$ ($m\ge 1$),  and 	$\widetilde{\Phi}$ being  the dual form of $\Phi$. We pick an orthogonal basis  $(\phi_{j})_{j=1}^{\infty}$  of  \textbf{even}  Hecke-normalized  Maass cusp forms of $SL_{2}(\Z)$ with $\Delta\phi_{j}=  ( 1/4-\mu_{j}^2) \, \phi_{j}$ and\,   $\lambda_{j}(a)$  ($a\ge 1$) being the Hecke eigenvalues of $\phi_{j}$.  

Let  $L\left( s , \phi_{j} \otimes  \Phi \right)$ and $L\left( s, \Phi \right)$ be the Rankin-Selberg $L$-function of  the pair $(\phi_{j}, \Phi)$ and the standard $L$-function of $\Phi$ respectively. We use the notation $\Lambda(\cdots)$ to denote the  completed $L$-function of $L(\cdots)$. The Hurwitz $\zeta$-function and the additively-twisted $L$-function of $\Phi$ are denoted by  $\zeta(s,a)$ and  $L(s, a/c; \Phi)$ respectively,  and define 	
\begin{align}\label{duamoadd}
		\mathcal{L}^{(a)}_{\pm}\left(s_{0}, s; \Phi \right)  \ \ := \ \    \sum_{dr \mid a} \   \frac{d^{2s_{0}-1} \mu(r)}{ r^{s_{0}} }\ \overline{\lambda_{\Phi}}\left( \frac{a}{dr}\right)  \  \sideset{}{^*}{\sum}_{\ell \ (\bmod\, d)} \  \zeta\left( 2s-s_{0}, \frac{\ell}{d}\right) 	L\left( s_{0}; \mp \frac{r\overline{\ell}}{d}; \Phi\right). 
\end{align}
See Section \ref{autoLfunc} for details.


\begin{thm}\label{twisdirc}

\noindent For  \  $\frac{1}{4}+ \frac{1}{200}  <  \sigma  < \frac{3}{4}$ and $H\in \mathcal{C}_{\eta}$, the following reciprocity formula holds:
\begin{align}\label{dirimom}
\hspace{15pt} 	   \sum_{j=1}^{\infty} \  H\left(\mu_{j}\right)  & \   \frac{\lambda_{j}(a)\Lambda(s, \phi_{j}\otimes \widetilde{\Phi}) }{\langle \phi_{j}, \phi_{j}\rangle}  \ + \   \int_{(0)} \   H\left(\mu\right) 
\frac{ \sigma_{-2\mu}(a) a^{-\mu}\Lambda( s+\mu,   \widetilde{\Phi} )\Lambda\left( 1-s+\mu,   \Phi \right)}{\left|\Lambda(1+2\mu)\right|^2} \ \frac{d\mu}{4\pi i} \nonumber\\
&\nonumber\\[-0.2in]
 &    \ = \  \frac{ a^{-s}}{2}  \ L(2s, \Phi) \ \prod_{p\mid a} \ \left\{  \overline{\lambda_{\Phi}}\left(p^{o_{p}(a)}\right) \ - \   \frac{\overline{\lambda_{\Phi}}\left(p^{o_{p}(a)-1}\right) }{p^{2s}} \right\} \nonumber\\
&\hspace{40pt} \  \cdot  \  	   \int_{(0)}  \frac{H(\mu)}{|\Gamma(\mu)|^2} \ \prod_{\pm} \ \prod_{i=1}^{3} \ \Gamma_{\R}\left(s \pm \mu - \alpha_{i}\right) \ \frac{d\mu}{2\pi i} \nonumber\\
&\nonumber\\[-0.2in]
&\hspace{80pt}  \ + \  \frac{a^{s-1}}{2}  \    L( 2(1-s), \widetilde{\Phi})  \    \prod_{p\mid a} \ \left\{   \lambda_{\Phi}\left(p^{o_{p}(a)}\right) \ - \   \frac{\lambda_{\Phi}\left(p^{o_{p}(a)-1}\right) }{p^{2(1-s)}} \right\} \nonumber\\
& \hspace{120pt} \ \cdot \   \int_{(0)} \ \frac{H(\mu)}{\left|\Gamma(\mu)\right|^2} \ \prod_{\pm} \ \prod_{i=1}^{3} \ \Gamma_{\R}\left(1-s \pm \mu + \alpha_{i}\right) \ \frac{d\mu}{2\pi i} \nonumber\\
&\nonumber\\[-0.1in]
&\hspace{150pt} \ + \  \frac{a^{-s}}{2}  \   \sum_{\pm} \  \int_{(1/2)} \  	\mathcal{L}^{(a)}_{\pm}\left(s_{0}, s; \Phi \right)   (\mathcal{F}_{\Phi}^{(\pm)}H)\left(s_{0},  s\right) \ \frac{ds_{0}}{2\pi i},
\end{align}
where\,  $\sigma_{-2\mu}(a) :=   \sum_{d \mid a} d^{-2\mu}$,  $o_{p}(a)$ denotes the power of $p$ in the prime factorization of $a$,  and  $(\mathcal{F}_{\Phi}^{(\pm)} H)\left(s_{0},  s \right)$ are certain explicit integral transforms to be defined in Section \ref{Stirl}.

\end{thm}

 Our theorem holds for \textit{general} twists of  $GL(2)$ Hecke eigenvalues  (i.e., $a$ is not necessarily  prime). In this instance, we find it more convenient to describe the dual moment of (\ref{dirimom}) in terms of  additively-twisted $L$-functions and this actually explains the origin of our reciprocity better (see Remark \ref{reas}).   In a number of previous works on twisted moments, e.g., \cite{Liu10, C07, Be16, Y11b,  AK18}, one had to restrict to   Hecke eigenvalue twists at primes  for various technical reasons.

 When $a=p$ is prime,  the dual moment can be expressed in terms of primitive multiplicatively-twisted $L$-functions as in (\ref{basicmototwis}) (see Proposition \ref{cuspdual}). Furthermore, if  $s=1/2$\,  and   $\Phi$ is replaced by the (complete) minimal parabolic Eisenstein series $(E_{\min}^{(3)})^{^*}( \, * \, ; \, (0,0,0))$ of $SL_{3}(\Z)$,  then the dual moment of (\ref{dirimom}) is precisely the  fourth moments of the Dirichlet $L$-functions:
 \begin{align}\label{youfourmo}
 \hspace{10pt} 	\frac{p^{\frac{1}{2}}}{\phi(p)} \  \ \   \sideset{}{^\pm}{\sum}_{\chi \ (\bmod\, p)} \ 	\int_{(1/2)} \   \  \left| L(s_{0}, \chi)\right|^4\,  (\mathcal{F}_{\alpha}^{(+)}H \pm  \mathcal{F}_{\alpha}^{(-)}H)\left(s_{0}, \, 1/2\right)  \frac{\Gamma_{\R}\left(1+ \frac{1\mp 1}{2}-s_{0}\right)}{\Gamma_{\R}\left(s_{0}+ \frac{1\mp 1}{2}\right)}  \ \frac{ds_{0}}{2\pi i},
 \end{align}
 which we will show in  Section \ref{rel4thDIr}.  As a matter of fact, the use of periods unveils that the source of the arithmetic (or non-archimedean) reciprocity in (\ref{youfourmo}) lies in a strikingly simple matrix identity: 
 \begin{align}\label{simpldec}
 		\begin{pmatrix}
 			* & * \\
 			a_{0} & p
 		\end{pmatrix}
 		\begin{pmatrix}
 			y_{0} & \\
 			           & 1
 		\end{pmatrix} \ =  \ \begin{pmatrix}
 		1 & -\overline{a_{0}}/p \\
 		   & 1
 		\end{pmatrix} n(x)a(y)k(\theta) \begin{pmatrix}
 		r & \\
 		   & r
 		\end{pmatrix},
 \end{align}
where the first matrix of (\ref{simpldec}) belongs to $SL_{2}(\Z)$, and the matrices $n(x)$, $a(y)$, $k(\theta)$, $\begin{psmallmatrix}
 r & \\
 & r
 \end{psmallmatrix}$ (as in the Iwasawa decomposition) are solely  responsible for the analytic (or archimedean) reciprocity. See Section \ref{Basicide}  for the details.

 When $a$ is composite, the relevant conversion remains valid but this time (\ref{youfourmo}) would average over all Dirichlet characters\,  $(\bmod\, a)$ --- both  primitive and imprimitive ones. Judging from the experiences of \cite{BHKM20} and \cite{Wu22}, this phenomenon should not be surprising even to  period integral approaches.



\subsection{Twisting and the `Recipe'}\label{twisintro}

Theorem \ref{twisdirc} exhibits a number of interesting  technical similarities with the asymptotic formulae for the  twisted moments of the Riemann  $\zeta$-function and  the Dirichlet $L$-functions.    There is extensive literature on this classical   topic:  Levinson  \cite{Le74}, Deshouillers-Iwaniec \cite{DI84}, Balasubramanian-Conrey-Heath-Brown \cite{BCHB85}, Conrey \cite{C89}, Watt \cite{W95} and Iwaniec-Sarnak \cite{IS99} --- just to name a few.

For more basic applications (say  \cite{Le74},  \cite{Y10},  \cite{BCHB85}, \cite{IS99}) involving    twisted  moments of lower degrees such as 
\begin{align}\label{twisze}
	\int_{T}^{2T} \ \left(\frac{h}{k}\right)^{it}\zeta(1/2+\alpha+it) \zeta(1/2+\beta-it)  \ dt,
\end{align}
or
\begin{align}\label{twisL}
\frac{1}{\phi(q)}	\ \ \sideset{}{^*}{\sum}_{\chi \ (\bmod\, q)} \ \chi(h) \overline{\chi}(k) L\left(1/2+ \alpha, \chi\right) L\left(1/2+ \beta, \overline{\chi}\right),
\end{align}
for any $h,k\ge 1 $ and $\alpha, \beta \ll (\log qT)^{-1}$,  the diagonals usually dominate the asymptotic evaluations  upon inserting the approximate functional equations (AFEs).      However, the twisted variant of the \textit{Moment Conjecture} (or   \textit{`twisted recipe'}) of \cite{CFKRS05}  does predict the existence of \textit{off-diagonal} contributions which are  non-negligible for more refined applications,  see   Bettin-Chandee-Radziwiłł \cite{BCR17}, Conrey-Iwaniec-Soundararajan \cite{CIS19}, Bui-Pratt-Robles-Zaharescu  \cite{BPRZ20} and Conrey \cite{C07}.  In  these works, the search for the essential off-diagonal terms turned out to be extremely subtle.  The  \textit{`shifts'} (i.e., $\alpha, \beta$ in (\ref{twisze})-(\ref{twisL})) in  the twisted moments are crucial in unraveling the combinatorics  behind the assembling of various distinct-looking terms. 
 


A similar situation is encountered in the context of Theorem \ref{twisdirc}.  In view of the applications of  Liu  \cite{Liu10} and  Khan  \cite{Kh15} regarding the    $GL(3)\times GL(2)$ $L$-values,  it was sufficient to extract only the  diagonal as the main term. As in Li \cite{Li11},   their moments were expanded as sums of  geometric terms via the AFEs and  Petersson/ Kuznetsov/ $GL(3)$ Voronoi formulae.  Unfortunately, it is  not clear from their works the locations of the off-diagonal main terms. 
A recurring theme in the area, as exemplified by  \cite{CIS12, CIS19, C07, BCR17},  involves  the identification of specific polar divisors subsequent to the re-packaging of the off-diagonal pieces.  This procedure is usually  delicate and depends heavily on the \textit{arithmetic} of the moment problem.  In Section \ref{ODMt}, we shall observe   the \textit{Motohashi-type structure}  depicted in Theorem \ref{twisdirc}  is indeed very suitable for a  simple extraction of  the off-diagonal main term.  In particular,  the twisted sums of $L$-functions  (\ref{duamoadd})   admit a clear polar divisor $2s-s_{0}=1$.  Overall speaking, we seek the most  canonical and straightforward means  to reveal the desired structure. It appears to us that the use of period integrals is  favourable in this regard.

In  Theorem  \ref{twisdirc}, we  examine the case when  $\Phi$ is a cusp form of $SL_{3}(\Z)$. This is nevertheless adequate to illustrate  the intricacies of understanding the main terms for twisted moments while keeping our exposition to a reasonable length. In fact, our twisted moment is of degree $6$ and its off-diagonal main term corresponds precisely to  the  `\textit{$3$-swap term}',  according to   the set-up and terminology  of Conrey-Keating \cite{CK19, CK16, CK15a, CK15b, CK15c} and Conrey-Fazzari \cite{CF23},   of the twisted cubic moment of $GL(2)$ $L$-functions:
 \begin{align}\label{twiscub}
	\sum_{j=1}^{\infty} \ H(\mu_{j}) \  \frac{ \lambda_{j}(a)\Lambda(1/2- \alpha_{1}, \phi_{j})\Lambda(1/2- \alpha_{2}, \phi_{j})\Lambda(1/2- \alpha_{3}, \phi_{j})}{\langle \phi_{j}, \phi_{j}\rangle}  \ + \ (\text{cont.}), 
\end{align}
which is simply the Eisenstein case of Theorem \ref{twisdirc}, i.e., taking  $\Phi$ to be $(E_{\min}^{(3)})^{^*}( \ * \ ; \, (\alpha_{1}, \alpha_{2}, \alpha_{3}))$. 

Our current knowledge on the   \textit{Dirichlet polynomial} approximations of $L$-functions enables us to extract the `$0$-swap' and  `$1$-swap' terms across various moments and  families, see \cite{HN22, Ng21, BTB22+, CK15c}, \cite{CF23} and \cite{CR22+}  for, respectively, the  unitary, orthogonal and symplectic examples. To go  beyond the $0$-swap and $1$-swap terms, one would require extra arithmetic  structures,  say the Motohashi-type structure present in our case.  Furthermore,  since our method is based on  period integrals,   we are able to  sidestep the technical complications stemming from the use of AFEs and treating the  $J$- and  $K$-Bessel pieces of the spectral Kuznetsov formula separately (see  \cite{Kw23} for further discussions).


Given a twisted moment of $L$-functions,  showing agreement with the relevant twisted recipe  often requires additional symmetries besides  the functional equations of  the $L$-functions. For instance,  certain \textit{finite Euler products}, which account for the ramifications coming  from twisting, satisfy  unexpected  functional relations. Within the unitary families, readers are referred to   \cite{HY10} (Section 6.2-6.3),  \cite{CIS19} (Section 8) and  \cite{BTB22+} (Section 9.3) (cf.  \cite{CK15c} and  \cite{HN22} (Section 5-6)) for the examples.   Established in a prime-by-prime (or \textit{local}) fashion with leverage on the delicate combinatorics of divisors,    these relations result in remarkable combinations and cancellations of terms,  ultimately leading to  the desired  prediction up to an acceptable error.    In the case of Theorem  \ref{twisdirc}, one has to unravel the complicated expression  (\ref{badEPexp}) for the hidden functional relations if one follows the standard approach in the literature. 
 
 In Section \ref{ODMt} of this article, we  proceed \textit{globally} instead by studying the  \textit{triple Dirichlet series}:
 \begin{align}\label{tripDS}
 	\sum_{a=1}^{\infty} \ a^{-w} \ \sum_{d \mid a} \  d^{4s-3} \
 	\sum_{a_{1}= 1}^{\infty}  \ \frac{\mathcal{B}_{\Phi}\left( a/d, a_{1}   \right)}{ (a_{1})^{2s-1}} \cdot S(0, a_{1};d),
 \end{align}
 where $\mathcal{B}_{\Phi}(\,\cdot\, , \, \cdot\, )$'s are the $GL(3)$ Fourier coefficients of $\Phi$ and $S(0, a_{1};d)$ is the Ramanujan sum.  We will explain how (\ref{tripDS}) arises from our period integral method as well as how it takes up  the role of  extra functional relations as in  \cite{HY10, CIS19, HN22,  BTB22+}.     The (multiple) Dirichlet series of $GL(3)$ will provide convenient packaging of  the underlying combinatorics. From which,   the recipe prediction is confirmed with an exact reciprocity identity and for arbitrary twists of Hecke eigenvalues of $GL(2)$.  In many ways,  the content of Section \ref{ODMt} serves as a non-archimedean analogue of the Barnes-type integral identities found in  \cite{Kw23a+}.  This kind of analogy was also investigated by Bump  \cite{Bu84} but in a different situation. Bump constructed an interesting double Dirichlet series  (see equation (9.2) of \cite{Bu84}) that mirrors  the Vinogradov-Takhtadzhyan  formula (see equation (10.1) of \cite{Bu84}).

 In a separate article \cite{Kw23a+}, we  tackled the \textit{analytic} issues for the untwisted  Eisenstein case of Theorem \ref{twisdirc}, including explication of the integral transforms,  regularization and analytic continuation (with  the Fourier expansion for the  $GL(3)$ Eisenstein series).   By combining the techniques of \cite{Kw23a+} with the  \textit{non-archimedean}  ingredients of this article, 
 the twisted recipe for (\ref{twiscub})  (see Conjecture \ref{twiscubconj}) should follow with some extra book-keeping.  It is noteworthy that the majority of our arguments work equally well for the Eisenstein case. 
 
  Both Theorem \ref{twisdirc} and Conjecture \ref{twiscubconj} express their main terms  in a local fashion. This better displays the symmetries ---  the finite part of  each term now consists of a product of $L$-values and   a `modified' Hecke eigenvalue of $GL(3)$.  As a matter of fact, this is a somewhat intriguing phenomenon for the off-diagonal main terms because  the  $GL(3)$ Hecke combinatorics are exploited in rather non-trivial ways. In a broader context,  it might be of interest to study the  twisted recipes for different families (perhaps also revisiting some of the previous works) by  searching for  suitable candidates of  multiple Dirichlet series and their combinatorial relations.





\subsection{Related Works}
In contrast to the Motohashi-type reciprocities which relate two  different-looking moments of $L$-functions,   there are also the recent  `\textit{level reciprocities}' of   Blomer-Khan \cite{BK19a, BK19b} roughly taking the shape
\begin{align}\label{levre}
		\sum_{\substack{f: \text{ level $p$ } \\ \text{cusp forms}}} \   \lambda_{f}(q)  L(1/2, \Phi \otimes  f)L(1/2,f) \ \  \longleftrightarrow \ \ 	\sum_{\substack{f: \text{ level $q$ } \\ \text{cusp forms}}} \   \lambda_{f}(p)  L(1/2, \widetilde{\Phi} \otimes  f)L(1/2,f)
\end{align}
for $p, q$ being primes. In the latter case,   two \textit{similar-looking} moments are connected,  akin to an earlier example  due to Conrey  \cite{C07} (see also  \cite{Y11b, Be16, DN22+}):
\begin{align}\label{Crep}
\hspace{-10pt} 	\sideset{}{^*}{\sum}_{\chi\, (\bmod\, p)} \ \chi(q) |L(1/2, \chi)|^2 \ \  \longleftrightarrow \ \  	\sideset{}{^*}{\sum}_{\chi\, (\bmod\, q)} \ \chi(-p) |L(1/2, \chi)|^2. 
\end{align}
The level reciprocity (\ref{levre})  features the switching of two distinct twists of  $GL(2)$ Hecke eigenvalues that appeared on both sides of the formula. It is conceivable that a certain invariant automorphic linear functional  which incorporates two distinct $GL(2)$  Hecke operators  is responsible for  such kind of reciprocities.  This was nicely confirmed by two different constructions of such functionals, see Zacharias \cite{Za21, Za20+} (for the `$4=2+2$' case)  and  Nunes \cite{Nu23}  (for the `$4=3+1$' case), and for Hecke twists at two coprime square-free integral ideals. 


While the interchange of harmonics $\left(\lambda_{f}(p)\right) \leftrightarrow  \left(\chi  \bmod\, p\right)$ described in (\ref{dirimom})-(\ref{youfourmo}) and (\ref{basicmototwis}) is also of non-archimedean nature (and to some extent a hybrid of the ones in (\ref{levre}) and (\ref{Crep})),  its primary  arithmetic cause has  less to do with the actions of the  Hecke operators  and is actually rather  different from   \cite{Za20+, Za21, Nu23}, see Remark \ref{reas}.

Moreover,  Theorem \ref{twisdirc} is  a higher-degree and  higher-rank analogue of the famous result of  Iwaniec-Sarnak (Theorem 17 of \cite{IS00}), Motohashi (Equation (2.33)-(2.36) of \cite{Mo92}) and  Bykovskii (\cite{By96}) which  has found plenty of applications in the area, see \cite{Mo92}, \cite{Iv01, Iv02}, \cite{J99}, \cite{Fr20}, \cite{IS00}, \cite{Liu18}, \cite{BF21}.





\section{Preliminary}\label{prel}

 In this section, we  collect results and notions  that are essential to our arguments. Some of which have appeared in \cite{Kw23} (Section \ref{prelimwhitt}-\ref{autformgl2}).

  \subsection{Notations and Conventions}\label{notconv}
  
  	Throughout this article,  $\Gamma_{\R}(s):= \pi^{-s/2}\, \Gamma(s/2)$  $(s\in \C)$,  $e(x):=e^{2\pi i x}$  $(x\in \R)$, and $\Gamma_{n}:=SL_{n}(\Z)$  $(n\ge 2)$.   Our test function $H$ lies in the class $ \mathcal{C}_{\eta}$ and $H:= h^{\#}$ (see Definition \ref{defwhittrans}) unless otherwise specified.  Denote by $\theta $ the best progress towards the Ramanujan conjecture for the Maass cusp forms of $SL_{3}(\Z)$.  We have $\theta\le \frac{1}{2}- \frac{1}{10}$, see Theorem 12.5.1 of \cite{Go15}.    We often use the same symbol to denote a function (in $s$) and its analytic continuation.  Also, we adopt the following sets of conventions:
  \begin{enumerate}
  	\item \label{hecke} All Maass cusp forms will be  simultaneous eigenfunctions of the Hecke operators and will be either even or odd. Also, their first Fourier coefficients are  equal to $1$. In this case, the forms are said to be \textbf{Hecke-normalized}. Note that there are no odd form for  $SL_{3}(\Z)$, see Proposition 9.2.5 of \cite{Go15}. 
  	

  	\item Our fixed Maass cusp form  of $SL_{3}(\Z)$ is assumed to be \textbf{tempered at $\infty$}, i.e., its Langlands parameters are purely imaginary.

  \end{enumerate}
  

\subsection{Whittaker Functions \& Transforms}\label{prelimwhitt}
All of the Whittaker functions  in this article are spherical.  The Whittaker function of $GL_{2}(\R)$ is given by 
 \begin{align}\label{gl2wh}
 	W_{\mu}(y) \  := \ 2\sqrt{y} K_{\mu}(2\pi y)
 \end{align}
 for $\mu\in \C$ and $y>0$.   Denote by $  W_{\alpha}\begin{psmallmatrix}
 	y_{0}y_{1} &          &       \\
 	& y_{0} &       \\
 	&          &  1
 \end{psmallmatrix}$ the Whittaker function of $GL_{3}(\R)$, where 
 \begin{align}\label{toricond}
 \hspace{15pt} 	y_{0},\, y_{1}>0 \hspace{15pt} \text{ and } \hspace{15pt} \alpha\in\mathfrak{a}_{\C}^{(3)} \ := \  \left\{ (\alpha_{1},\alpha_{2}, \alpha_{3})\in \C^{3}: \alpha_{1}+\alpha_{2}+\alpha_{3}=0 \right\}. 
 \end{align}
To define $  W_{\alpha}$,  we first introduce the function  
 \begin{align}
 \hspace{-40pt} 	I_{\alpha}\begin{psmallmatrix}
 		y_{0}y_{1} &          &       \\
 		& y_{0} &       \\
 		&          &  1
 	\end{psmallmatrix} \ := \ y_{0}^{1-\alpha_{3}} y_{1}^{1+\alpha_{1}}
 \end{align}
whenever  (\ref{toricond}) holds. Then we define
\begin{align}\label{jac}
		W_{\alpha}\begin{psmallmatrix}
			y_{0}y_{1} &          &       \\
			& y_{0} &       \\
			&          &  1
		\end{psmallmatrix} \  := \  &
 \prod_{1\le j<k\le 3} \, \Gamma_{\R}(1+\alpha_{j}-\alpha_{k}) \nonumber\\
 &\hspace{20pt} \cdot \,  \int_{\R^3} \   I_{\alpha}\left[\begin{psmallmatrix}
	&    & 1\\
	& -1 &    \\
	1 &  &
	\end{psmallmatrix} \begin{psmallmatrix}
	1 & u_{1,2} & u_{1,3} \\
	&     1      & u_{2,3} \\
	&             & 1
	\end{psmallmatrix} 
	\begin{psmallmatrix}
	y_{0}y_{1} &          &       \\
	& y_{0} &       \\
	&          &  1
	\end{psmallmatrix}\right] \, e(-u_{1,2}-u_{2,3}) \,  \prod_{1\le j<k\le 3}  \, du_{j,k},
	\end{align}
	once again for  (\ref{toricond}). The last integral is known as the \textit{Jacquet integral}.   See Chapter 5.5 of \cite{Go15} for details.  We recall the explicit evaluation of the Rankin-Selberg integral of $GL_{3}(\R)\times GL_{2}(\R)$.

 \begin{prop}\label{stadediff1}
  Let  $W_{\mu}$ and\,  $W_{-\alpha}$  be the Whittaker functions of  $GL_{2}(\R)$ and $GL_{3}(\R)$  respectively.  For $\re s\gg 0$, we have
 	\begin{equation}\label{eqn sta1}
	\int_{0}^{\infty} \int_{0}^{\infty} \  W_{\mu}(y_{1}) \, W_{-\alpha}\begin{psmallmatrix}
		y_{0}y_{1} &          &       \\
		& y_{0} &       \\
		&          &  1
	\end{psmallmatrix}\, (y_{0}^2 y_{1})^{s-\frac{1}{2}} \ \frac{dy_{0} dy_{1}}{y_{0}y_{1}^2} \ = \  \frac{1}{4} \,   \prod_{\pm} \prod_{k=1}^{3} \ \Gamma_{\R}\left(s\pm \mu-\alpha_{k}\right). 
 	\end{equation}
 \end{prop}

 \begin{proof}
 See  \cite{Bu88}.
 \end{proof}


As in \cite{Kw23, Kw23a+}, the following pair of integral transforms plays an important role in the description of the archimedean component of our main theorem. 

	\begin{defi}\label{defwhittrans}
	Let  $h: (0, \infty) \rightarrow \C$ and $H: i\R \rightarrow \C$ be  measurable functions with $H(\mu)=H(-\mu)$. Then the Kontorovich-Lebedev transform of $h$ is defined by
	\begin{equation}\label{eqn whitranseq}
	h^{\#}(\mu) \ := \ \int_{0}^{\infty} h(y)\, W_{\mu}(y) \ \frac{dy}{y^2},   
	\end{equation}
	whereas its inverse transform is defined by 
	\begin{align}\label{invers}
	H^{\flat}(y) \ :=  \ \frac{1}{4\pi i} \int_{(0)}  \ H(\mu)\, W_{\mu}(y) \ \frac{d\mu}{\left| \Gamma(\mu)\right|^2},
	\end{align}
	provided  the integrals converge absolutely. 	
	\end{defi}

		\begin{defi}
				Let $\mathcal{C}_{\eta}$ be the class of holomorphic functions $H$  on the vertical strip  $|\re \mu|< 2\eta$ such that 
				\begin{enumerate}
					\item $H(\mu)=H(-\mu)$,
					
					\item $H$ has rapid decay in the sense that
					\begin{align}\label{rapdeca}
					H(\mu) \ \ll \ e^{-2\pi|\mu|} \hspace{50pt} (|\re \mu| \ < \  2\eta). 
					\end{align}
				\end{enumerate}
				In this article, we take $\eta>40$ without otherwise specifying. 
			\end{defi}

	\begin{prop}\label{inKLconv}
	For any $H\in \mathcal{C}_{\eta}$, the integral  (\ref{invers}) defining $H^{\flat}$   converges absolutely. Moreover, we have 
				\begin{align}\label{decbdd}
			    H^{\flat}(y) \ \ll \ \min \{y, y^{-1}\}^{\eta} \hspace{20pt} (y \ > \ 0). 
				\end{align}

		\end{prop}
		
		\begin{proof}
See Lemma 2.10 of \cite{Mo97}.
			\end{proof}

	\begin{prop}\label{plancherel}
		Under the same  assumptions of Proposition \ref{inKLconv}, we have
		\begin{align}
		(h^{\#})^{\flat}(g) \ = \ h(g)  \hspace{20pt} \text{ and } \hspace{20pt} 
		(H^{\flat})^{\#}(\mu) \ = \ H(\mu). 
		\end{align}
	\end{prop}
	
	\begin{proof}
		See Lemma 2.10 of \cite{Mo97}.
	\end{proof}


\subsection{Automorphic Forms of $GL(2)$}\label{autformgl2}
Let 
 \begin{align*}
 	\mathfrak{h}^2 \ := \ \left\{ \, g \ = \  \begin{pmatrix}
 		1 & u\\
 		& 1
 	\end{pmatrix} \begin{pmatrix}
 		y & \\
 		& 1
 	\end{pmatrix} : u \in \R, \ y >  0\,  \right\}
 \end{align*}
 with its invariant measure  given by $dg:=y^{-2} du \, dy$. We use  \ $\langle \ \cdot \ ,\  \cdot \ \rangle$ \ to denote the Petersson inner product on $\Gamma_{2}\setminus \mathfrak{h}^{2}$, i.e., 
 \begin{align*}
 	\left\langle \ \phi_{1} \, ,\,  \phi_{2} \ \right\rangle \ \ := \ \  \int_{\Gamma_{2}\setminus \mathfrak{h}^{2}} \ \phi_{1}(g) \, \ovA{\phi_{2}(g)} \ dg.
 \end{align*}

Let $\Delta:= -y^2\left( \partial_{x}^2 +\partial_{y}^2\right)$. An automorphic form  $\phi:  \mathfrak{h}^2 \rightarrow \C$ of $\Gamma_{2}=SL_{2}(\Z)$ satisfies  $\Delta\phi \,= \, \left( \frac{1}{4} -\mu^2\right) \phi$ for some $\mu= \mu(\phi) \in \C$. We often identify $\mu$ with the pair $(\mu, -\mu)\in \mathfrak{a}_{\C}^{(2)}$. For $a\in \Z-\{0 \}$, the $a$-th Fourier coefficient of $\phi$, denoted by $\mathcal{B}_{\phi}(a)$, is defined  by
\begin{equation}\label{eiscuspcoeff}
(\widehat{\phi})_{a}(y) \ := \  \int_{0}^{1} \ \phi\left[ \begin{pmatrix}
1 & u\\
& 1
\end{pmatrix} \begin{pmatrix}
y & \\
& 1
\end{pmatrix}\right] e(-au) \ du \ = \  \frac{\mathcal{B}_{\phi}(a)}{\sqrt{|a|}} \,  W_{\mu(\phi)}(|a|y) \hspace{15pt} (y\ > \ 0).
\end{equation}

Let   $I_{\mu}(y):= y^{\mu+\frac{1}{2}}$. The Eisenstein series of $\Gamma_{2}$ is given by
		\begin{align}\label{gl2eins}
		E(z; \mu) \ := \  \frac{1}{2} \  \sum_{\gamma\in U_{2}(\Z)\setminus\Gamma_{2}} I_{\mu}(\im \gamma z) \hspace{20pt} (z \ \in \  \mathfrak{h}^2).
		\end{align}
	It is well-known that the series (\ref{gl2eins}) converges absolutely for  $\re \mu>1/2$  and admits a meromorphic continuation to the whole complex plane.  Also, we have $\Delta E(\, *\, ;  \mu) \,  = \,  \left( \frac{1}{4}-\mu^2\right) E(\, *\, ;  \mu)$ and  the Fourier coefficients  of $E(*;\mu)$ are given by  
		\begin{align}\label{eisfournorm}
		\mathcal{B}(a;\mu) \ = \  \mathcal{B}_{E(\, * \, ; \, \mu)}(a)    \ = \  \frac{|a|^{\mu}\sigma_{-2\mu}(|a|)}{\Lambda(1+2\mu)}, 
		\end{align}	  
		where
		\begin{align*}
			\Lambda(s) \ := \  \pi^{-s/2}\, \Gamma(s/2)\zeta(s) \hspace{15pt} \text{ and } \hspace{15pt}  \sigma_{-2\mu}(|a|) \ := \  \sum_{d \mid a} d^{-2\mu}.  
		\end{align*}
		 We also write $E^{*}(*;\mu):=\Lambda(1+2\mu)\, E(*;\mu)$\, for the complete Eisenstein series of $\Gamma_{2}$  and    $\mathcal{B}_{E^{^*}(\, * \, ; \, \mu)}(a)$ for its Fourier coefficients.

\subsection{Automorphic Forms of $GL(3)$}\label{autformgl3}

Next, let  
\begin{align*}
	\mathfrak{h}^{3} \ := \  \left\{  \begin{pmatrix}
		1 & u_{1,2} & u_{1,3} \\
		&     1      & u_{2,3} \\
		&             & 1
	\end{pmatrix} 
	\begin{pmatrix}
		y_{0}y_{1} &          &       \\
		& y_{0} &       \\
		&          &  1
	\end{pmatrix}: u_{i,j} \in \R, \ y_{k} >0 \right\}.  
	\end{align*}
 Let $\Phi: \mathfrak{h}^3 \rightarrow \C$ be a Maass cusp form of $\Gamma_{3}$ as defined in Definition 5.1.3 of \cite{Go15}. In particular,  there exists  $\alpha= \alpha(\Phi)\in \mathfrak{a}_{\C}^{(3)}$ such that  for any $D\in Z(U\mathfrak{gl}_{3}(\C))$  (i.e., the center of the universal enveloping algebra of the Lie algebra $\mathfrak{gl}_{3}(\C)$),  we have
\begin{align*}
D\Phi \ = \  \lambda_{D} \Phi \hspace{20pt} \text{ and } \hspace{20pt} 
DI_{\alpha} \ = \ \lambda_{D} I_{\alpha}
\end{align*}
for some $\lambda_{D}\in \C$. The triple $\alpha(\Phi)$ is said to be  the \textit{Langlands parameters} of $\Phi$. 

\begin{defi}\label{fourcoeff}
	Let $m=(m_{1}, m_{2})\in (\Z-\{0\})^{2}$ and $\Phi:  \mathfrak{h}^3 \rightarrow \C$ be a Maass cusp form of $SL_{3}(\Z)$. 	The integral defined by
	\begin{align}\label{fourcoeff2}
	\hspace{10pt} (\widehat{\Phi})_{(m_{1}, m_{2})}(g)\ &:= \ \int_{0}^{1} \int_{0}^{1}\int_{0}^{1} \Phi\left[ \begin{pmatrix}
	1 & u_{1,2} & u_{1,3} \\
	&     1      & u_{2,3} \\
	&             & 1
	\end{pmatrix} 
	g\right]   e\left(-m_{1} u_{2,3}-m_{2} u_{1,2} \right) \ du_{1,2} \ du_{1,3} \ du_{2,3},
	\end{align}
for $g\in GL_{3}(\R)$,  is said to be the  $(m_{1}, m_{2})$-th \textbf{Fourier-Whittaker period} of $\Phi$. 
Moreover,  the $(m_{1}, m_{2})$-th \textbf{Fourier coefficient} of $\Phi$ is the complex number $\mathcal{B}_{\Phi}(m_{1}, m_{2})$  for which
	\begin{align}\label{fourmulone}
		(\widehat{\Phi})_{(m_{1}, m_{2})}	\begin{pmatrix}
		y_{0}y_{1} &          &       \\
		& y_{0} &       \\
		&          &  1
		\end{pmatrix} \ = \  \frac{\mathcal{B}_{\Phi}(m_{1}, m_{2})}{|m_{1}m_{2}|} \  W_{\alpha(\Phi)}^{^{\sgn(m_{2})}}\begin{pmatrix}
		(|m_{1}|y_{0})(|m_{2}|y_{1}) &          &       \\
		& |m_{1}|y_{0} &       \\
		&          &  1
		\end{pmatrix}
	 \end{align}
     holds for any  $y_{0}, y_{1}>0$.   
\end{defi}

\begin{rem}\label{heckeigdef}
	\
	\begin{enumerate}
		\item The archimedean multiplicity-one theorem of Shalika  guarantees  the well-definedness of the Fourier coefficients for $\Phi$.  For a proof, see  Theorem 6.1.6 of \cite{Go15}. 
		
		\item We shall make use of the following simple observation frequently:
		\begin{align}\label{unipotran}
			(\widehat{\Phi})_{(m_{1}, m_{2})}\left[ \begin{pmatrix}
				1 & u_{1,2} & u_{1,3} \\
				&     1      & u_{2,3} \\
				&             & 1
			\end{pmatrix} 
			g\right] \  \ = \ \  e\left(m_{1}u_{2,3}+m_{2}u_{1,2}\right) \cdot 	(\widehat{\Phi})_{(m_{1}, m_{2})}(g)
		\end{align}
		for any $g\in GL_{3}(\R)$ and  $u_{i,j}\in \R$ \ ($1\le i<j\le 3$). 
		
	\end{enumerate}

\end{rem}

 If $\Phi$ is Hecke-normalized (part of the convention of this article, see Section \ref{notconv}.(\ref{hecke})), then $\mathcal{B}_{\Phi}(1,n)$ can be shown to be a Hecke eigenvalue of $\Phi$ and thus is a multiplicative function. In this case, we  write $\lambda_{\Phi}(n):= \mathcal{B}_{\Phi}(1,n)$.  Furthermore, we have
\begin{align}\label{splGL3Hec}
	\mathcal{B}_{\Phi}(a, a_{1}) \ = \  \sum\limits_{r \mid (a, a_{1})} \ \mu(r) \overline{\lambda_{\Phi}}(a/r) \lambda_{\Phi}(a_{1}/r)
\end{align}
for any $a, a_{1}\in \Z-\{0\}$, where $\mu(\cdot)$ denotes the M\"obius $\mu$-function.  See  Section 6.4 of \cite{Go15} for  details.


\subsection{Automorphic $L$-functions}\label{autoLfunc}
The Maass cusp forms $\Phi$ and  $\phi$  of  $\Gamma_{3}$ and $\Gamma_{2}$  are assumed to be Hecke-normalized.  Denote their Langlands parameters  by    $\alpha=\alpha(\Phi)\in \mathfrak{a}_{\C}^{(3)}$ and $\mu=\mu(\phi) \in \mathfrak{a}_{\C}^{(2)}$  respectively.   If $\widetilde{\Phi}(g):= \Phi\left( ^{t}g^{-1}\right)$ is the dual form of $\Phi$, then the Langlands parameters of $\widetilde{\Phi}$ are given by $-\alpha$.

	\begin{defi}\label{DScuspRS}
		Suppose $\Phi$  (resp. $\phi$) is a  Maass cusp form or an Eisenstein series of   $\Gamma_{3}$ (resp. $\Gamma_{2}$).  For $\re s \gg 1$, the Rankin-Selberg $L$-function of $\Phi$ and $\phi$ is defined by
	\begin{align}\label{rankse}
		L\left(s, \phi\otimes \Phi\right) \ &:= \    \sum_{m_{1}=1}^{\infty} \ \sum_{m_{2}= 1}^{\infty}  \frac{\mathcal{B}_{\phi}(m_{2})  \mathcal{B}_{\Phi}(m_{1}, m_{2} )}{ \left(m_{1}^2 m_{2}\right)^{s}}. 
	\end{align}
\end{defi}

		\begin{prop}\label{ranselmainthm}
		Suppose $\Phi$ and $\phi$ are  Maass cusp forms  of   $\Gamma_{3}$ and  $\Gamma_{2}$  respectively.  In addition, assume that $\phi$ is  even.  Then for	any $\re s\gg 1$, we have
		\begin{align}\label{sameparun}
		\left(\phi,\  \mathbb{P}_{2}^{3}\Phi\cdot |\det *|^{\overline{s}-\frac{1}{2}}\right) _{\Gamma_{2}\setminus GL_{2}(\R)} \ \ := \  \  \int_{\Gamma_{2}\setminus GL_{2}(\R)}  \ \phi(g) \widetilde{\Phi}\begin{pmatrix}
				g & \\
				& 1
			\end{pmatrix} |\det g|^{s-\frac{1}{2}} \ dg  \ = \ 	\frac{1}{2} \  \Lambda(s, \phi\otimes \widetilde{\Phi}),
		\end{align}
	where
		\begin{align}
			\Lambda(s, \phi\otimes \widetilde{\Phi}) \ &:= \ L_{\infty}(s, \phi\otimes \widetilde{\Phi})\, L(s, \phi\otimes \widetilde{\Phi})
		\end{align}
	and 
		\begin{align}
				L_{\infty}(s, \phi\otimes \widetilde{\Phi}) \ &:= \  \prod_{k=1}^{3} \ \Gamma_{\R}\left(s\pm  \mu- \alpha_{k}\right). 
		\end{align}

	\end{prop}

	\begin{proof}
 See \cite{Go15} and the relevant remarks in  \cite{Kw23}. 
	\end{proof}

	\begin{defi}\label{JPSSstdL}
		Let $\Phi: \mathfrak{h}^{3}\rightarrow \C$ be a Maass cusp form or Eisenstein series of $\Gamma_{3}$. For $\re s  \gg  1$,  the standard $L$-function of $\Phi$ is defined by
		\begin{align}\label{gl3stdL}
			L\left(s, \Phi \right)\ := \ \sum_{n=1}^{\infty} \ \frac{\mathcal{B}_{\Phi}(1,n)}{n^{s}}.   
		\end{align}
	\end{defi}

The standard $L$-function  admits an entire continuation and  satisfies a  functional equation: 
	
\begin{prop}\label{globgl3func}
	Let $\Phi: \mathfrak{h}^{3}\rightarrow \C$ be a Maass cusp form of $\Gamma_{3}$. For any $s\in \C$, we have
		\begin{align}\label{JPfunc}
		\Lambda\left(s, \Phi\right) \ = \ \Lambda(1-s,  \widetilde{\Phi}),
	\end{align}
	where 
	\begin{align}
	\Lambda\left(s, \Phi\right) \ := \ 	L_{\infty}\left(s, \Phi \right) \, L\left(s, \Phi\right)
	\end{align}
and
	\begin{align}
L_{\infty}\left(s, \Phi \right) \ := \  \prod_{k=1}^{3}  \ \Gamma_{\R}\left(s+\alpha_{k} \right). 
\end{align}

\end{prop}

	\begin{proof}
		See Chapter 6.5 of \cite{Go15}.
	\end{proof}

Furthermore,  the standard $L$-function $L(s, \phi)$ (resp. $L(s, \Phi)$) for  $\phi$ (resp. $\Phi$), which is either a Hecke-normalized Maass cusp form or a complete Eisenstein series of $\Gamma_{2}$ (resp.  $\Gamma_{3}$),  admits an Euler product of the form
\begin{align}\label{23Eul}
L(s, \phi) \ = \  \prod_{p} \ \prod_{j=1}^{2} \ \left(1-\beta_{\phi, j}(p)p^{-s}\right)^{-1}, 
\end{align}
resp. 
\begin{align}\label{3EUl}
	L(s, \Phi) \ = \  \prod_{p} \ \prod_{k=1}^{3} \ \left(1-\alpha_{\Phi, k}(p)p^{-s}\right)^{-1}
\end{align}
for $\re s \gg 1$. One can show that
\begin{align}\label{RSEul}
L(s, \phi \otimes \Phi) \ \ = \  \   \prod_{p} \ \prod_{j=1}^{2}  \ \prod_{k=1}^{3} \  \left(1- \beta_{\phi, j}(p)\alpha_{\Phi, k}(p)p^{-s}\right)^{-1},
\end{align}
see  Proposition 7.4.12 of \cite{Go15}.   If $\phi= E^{*}(*;\mu)$, then   (\ref{23Eul}) holds for $\left\{ \beta_{\phi,1}(p), \beta_{\phi,2}(p)  \right\} $ $ =  \{ p^{\mu}, p^{-\mu}  \}$ and  (\ref{23Eul})-(\ref{RSEul}) imply the Dirichlet series  $L(s, \phi\otimes \Phi)$ is given by the product of $L$-functions  $L(s+\mu, \Phi)L(s-\mu, \Phi)$.

	\begin{prop}\label{ranseleis}
	Suppose $\Phi$ is a Hecke-Maass cusp form of $\Gamma_{3}$. For any $s, \mu \in \C$, we have
	\begin{align}\label{samepareis}
		\hspace{-23pt}	\left(\, E^{*}(\, * \, ; \,  \mu)\,  , \, \left( \mathbb{P}_{2}^{3} \Phi\right) \cdot |\det *|^{\bar{s}-\frac{1}{2}}\right)_{\Gamma_{2}\setminus GL_{2}(\R)} 
		\ \ = \ \ \frac{1}{2}   \ \Lambda( s+ \mu,  \widetilde{\Phi})\Lambda( s- \mu,  \widetilde{\Phi}).
	\end{align}
	\end{prop}
	
	\begin{proof}
		Parallel to Proposition \ref{ranselmainthm}. 
		\end{proof}

We will also need the basic analytic properties of certain `generalized' zeta-functions and automorphic $L$-functions which are stated as follows. 

 \begin{prop}\label{simpvor}
	\
	\begin{enumerate}
		\item 	For $\re s>1$  and  $0<a\le 1$,  the Hurwitz $\zeta$-function, which is defined by 
		\begin{align}\label{hurw}
			\zeta(s, a) \ := \ \sum_{n=0}^{\infty} \, (n+a)^{-s}, 
		\end{align}
		admits a holomorphic continuation to $\C$ except at $s=1$. It has  a simple pole at $s=1$ and the residue is equal to  $1$. Moreover, it has polynomial growth in every vertical strip.

		\item Let $\Phi$ be a Maass cusp form of $\Gamma_{3}$. For $\re s>1+\theta$ and $a/c\in \Q$, the additively-twisted $L$-series of $\Phi$ by $a/c$, which is defined by
		\begin{align}\label{voro}
			L\left( s, \frac{a}{c}; \Phi\right) \ := \  \sum_{\substack{n=1} }^{\infty} \    \  \frac{\mathcal{B}_{\Phi}(1, n)}{ n^{s}} \,   e\left( \frac{na}{c}\right),
		\end{align}
		admits an entire continuation and it has polynomial growth in every vertical strip. 
	\end{enumerate}
\end{prop}

\begin{proof}
	See Theorem 12.4 of  \cite{Ap76} and Lemma 3.2 of  \cite{GL06}. 
\end{proof}

	It is well-known that both  (\ref{hurw}) and (\ref{voro}) satisfy some forms of  functional equations, see Theorem 12.6 of \cite{Ap76} and  Theorem 3.1 of \cite{GL06} respectively.   However, we will make no use of such formulae in the proof of Theorem \ref{twisdirc}. The  polynomial growth of  (\ref{hurw}) and (\ref{voro}) in vertical strips can be deduced from that of the Dirichlet $L$-functions and the multiplicatively-twisted $L$-functions of $\Phi$. Moreover, the regularity of our class of test functions  $\mathcal{C}_{\eta}$ is more than sufficient for all of our analytic purposes provided $\eta> 40$.

		
		\subsection{Calculations on the Spectral Side}

	\begin{defi}\label{poindef}
		Let $a\ge 1$ be an integer and $h\in C^{\infty}(0,\infty)$. The Poincar\'e series of $\Gamma_{2}$   is defined by
		\begin{equation}\label{defpoin}
		P^{a}(z; h) \ := \   \sum_{\gamma\in U_{2}(\Z)\setminus \Gamma_{2}} h(a \im \gamma z) \cdot e\left(a \re \gamma z\right) \hspace{20pt} (z \ \in \  \mathfrak{h}^2)
		\end{equation}
	 provided it converges absolutely. 
	\end{defi}
	
If the bounds 
	\begin{align}\label{poincbdd}
	h(y)  \ \ll \  y^{1+\epsilon} \hspace{10pt} \text{(as \ $y\to 0$)} \hspace{10pt} \text{ and }   \hspace{10pt} h(y) \  \ll \   y^{\frac{1}{2}-\epsilon}  \hspace{10pt} \text{ (as \ $y \to\infty$)}
	\end{align}
are satisfied, then the Poincar\'e series $P^{a}(z; h)$ converges absolutely and is an $L^2$-function on $\Gamma_{2}\setminus \mathfrak{h}^2$.  In this article, we  take $h:= H^{\flat}$ with $H\in \mathcal{C}_{\eta}$ and $\eta >40$.   The conditions of (\ref{poincbdd})  clearly hold because of Proposition \ref{inKLconv}.  We often use the shorthand $P^{a}:= P^{a}(*;h)$.

	\begin{prop}\label{comspec}
		Let $\Phi$ be a  Hecke-Maass cusp form  of $\Gamma_{3}$ and $P^{a}$ be a Poincar\'e series of $\Gamma_{2}$. Then for $s \in \C$, we have
		\begin{align}\label{specside}
		 &\hspace{-50pt} 2a^{-1/2}  \, \left(P^a, \ \mathbb{P}_{2}^{3} \Phi\cdot  |\det *|^{\overline{s}-\frac{1}{2}}\right)_{\Gamma_{2}\setminus GL_{2}(\R)}  \nonumber\\
	 \hspace{40pt} &	\ = \    \sum_{j=1}^{\infty} \  h^{\#}\left(\mu_{j}\right)\,  \frac{\ovB{\mathcal{B}_{j}(a)}  \, \Lambda(s, \phi_{j}\otimes \widetilde{\Phi}) }{\langle \phi_{j}, \phi_{j}\rangle} \nonumber\\
	&\hspace{40pt} \ + \   \int_{(0)} \   h^{\#}\left(\mu\right) \, 
	\frac{ \sigma_{-2\mu}(a) a^{-\mu} \Lambda( s+ \mu,   \widetilde{\Phi} )\Lambda( s- \mu,   \widetilde{\Phi})}{\left|\Lambda(1+2\mu)\right|^2} \ \frac{d\mu }{4\pi i },
		\end{align}
	where  the sum is restricted to an orthogonal basis $(\phi_{j})$ of even Hecke-normalized  Maass cusp forms for $\Gamma_{2}$ with $\Delta\phi_{j}= \left( \frac{1}{4}-\mu_{j}^2\right)\phi_{j}$ and $\mathcal{B}_{j}(a):= \mathcal{B}_{\phi_{j}}(a)$. 
		
	\end{prop}

	\begin{proof}
See  \cite{Kw23}. 
		\end{proof}
	
			
				\subsection{Archimedean Analysis:  Properties of the Integral Transforms $\mathcal{F}^{(\pm)}_{\Phi}$}\label{Stirl}

			Let $H=h^{\#}$. The integral transforms for this  article as well as \cite{Kw23, Kw23a+}  are given by
			\begin{align}\label{duaintrafir}
				\hspace{10pt}	(\mathcal{F}^{(\pm)}_{\Phi} H)\left(s_{0},  s \right) \ = \   \int_{0}^{\infty}  \int_{0}^{\infty} \  h\left( \frac{y_{1}}{\sqrt{1+y_{0}^2}}\right) \ & \frac{ y_{0}^{2s-s_{0}}}{(1+y_{0}^{2})^{\frac{s}{2}-s_{0}+\frac{1}{4}}}   \nonumber\\
				&	\hspace{-55pt} \cdot \,  \left( \int_{0}^{\infty}  \ W_{-\alpha(\Phi)}\begin{psmallmatrix}
					Xy_{1} & & \\
					            & X & \\
					            &     & 1
				\end{psmallmatrix}  e(\pm Xy_{0}) X^{s_{0}-1} \ d^{\times} X \right)
				y_{1}^{s-\frac{1}{2}}\   \  \frac{dy_{0}\, dy_{1}}{y_{0} y_{1}^2}. \nonumber\\ 
			\end{align}
			
			Readers should consult  \cite{Kw23} for various formulations of  (\ref{duaintrafir}) in terms of  Mellin-Barnes integrals. For technical reasons, it is often convenient to work with the `perturbed' version of  $(\mathcal{F}^{(\pm)}_{\Phi} H)\left(s_{0},  s \right) $, denoted by $(\mathcal{F}^{(\pm)}_{\Phi} H)\left(s_{0},  s;\, \phi \right) $  (for $\phi \in (0, \pi/2]$), see equation (7.6)-(7.7) of \cite{Kw23}.  We have $(\mathcal{F}^{(\pm)}_{\Phi} H)\left(s_{0},  s \right) = (\mathcal{F}^{(\pm)}_{\Phi} H)\left(s_{0},  s; \, \pi/2 \right) $. For the proof of (\ref{duaintrafir}), see Section \ref{suppInt}.

			As in \cite{Kw23, Kw23a+}, we take $\epsilon:=1/100$\,  throughout this article.

			\begin{prop}\label{anconpr}
				Suppose  $H\in \mathcal{C}_{\eta}$, $s:= \sigma+it$ and $s_{0}:= \sigma_{0}+it_{0}$. Then
				
				\begin{enumerate}
					
					\item For any $\phi \in (0, \pi/2]$, the transform   $(s_{0}, s)\mapsto (\mathcal{F}_{\Phi}^{(\pm)}H)(s_{0},\, s;\, \phi)$  is holomorphic on the domain
					\begin{align}\label{newdomain}
				  \mathcal{D} \ := \  \left\{  (\sigma_{0}, \sigma): \,    \sigma_{0} \, > \,  \epsilon,  \  \sigma \, < \, 4, \   2\sigma-\sigma_{0}-\epsilon \, > \, 0 \right\}.	
					\end{align}
					
					\item Whenever $(\sigma_{0}, \sigma) \in \mathcal{D}$,  $|t| <T$ and $\phi \in (0, \pi/2)$, the transform $(\mathcal{F}_{\Phi}^{(\pm)}H)\left(s_{0}, \, s; \, \phi\right)$ has exponential decay as  $|t_{0}| \to \infty$.   
					
					\item There exists a constant  $B=B_{\eta}$ such that  whenever $(\sigma_{0}, \sigma) \in 
					\mathcal{D}$, $|t| <T$, and $|t_{0}| \gg 1$,  we have the estimate
					\begin{align}\label{refbdd}
						\left|(\mathcal{F}_{\Phi}^{(\pm)}H)\left(s_{0}, \, s\right) \right| \ \ll \  |t_{0}|^{8-\frac{\eta}{2}}   \log^{B} |t_{0}|,
					\end{align}
					where the implicit constants depend only on $\eta$, $T$, $\Phi$. 
				\end{enumerate}
			\end{prop}
			
			\begin{proof}
				See Proposition 8.1 and Proposition 9.1 of  \cite{Kw23}.
			\end{proof}

			 	\begin{prop}\label{secMTcom}
				Suppose  $1/2  <  \sigma  < 1$. Then 
				\begin{align}\label{secMTcomid}
				\hspace{10pt} 	\sum_{\pm} \ (\mathcal{F}^{(\pm)}_{\Phi} H)\left(2s-1,  s \right) 
					\ = \ \pi^{\frac{1}{2}-s}  &\cdot \,  \prod_{i=1}^{3} \ \frac{\Gamma\left(s-\frac{1}{2}+ \frac{\alpha_{i}}{2}\right)}{\Gamma\left(1-s- \frac{\alpha_{i}}{2}\right)} \nonumber\\
					&\hspace{25pt}  \cdot \,  \int_{(0)} \  \frac{H(\mu)}{\left| \Gamma(\mu)\right|^2}   \cdot \prod\limits_{i=1}^{3}  \ \prod\limits_{\pm} \  \Gamma\left( \frac{1-s+ \alpha_{i}\pm \mu}{2}\right) \ \frac{d\mu}{2\pi i}. 
				\end{align}
			\end{prop}
			
			\begin{proof}
				See Theorem 1.2 of \cite{Kw23}. 
			\end{proof}


 \section{Twisted Key Identity}\label{Basicide}

  In this section, we illustrate how the arithmetic  (or non-archimedean) twists alluded to Section \ref{intr} enter the picture of our period integral approach. We must include a couple of technical adjustments to our argument in  \cite{Kw23} in order to  incorporate the desired new features.

 \begin{prop}\label{incomf}
 Let  $a\in \Z-\{0\}$  and $\Phi$ be a smooth automorphic function  of  $\Gamma_{3}$. Then for any $g\in GL_{3}(\R)$,  we have
 	\begin{align}\label{keylemform}
	\hspace{10pt}	\int_{0}^{1} \Phi\left[
	\begin{pmatrix}
		1 & u_{1,2} &  \\
		& 1          & \\
		&             & 1
	\end{pmatrix} g\right] \, &   e( -a u_{1,2}) \ du_{1,2}  \nonumber\\ 
	\ &= \ \   \sum_{d \mid a} \ 
	\sum_{\substack{a_{0} \in \Z  \\ \gcd(a_{0}, d)=1}} \ \sum_{a_{1}=-\infty}^{\infty} \  \left(\widehat{\Phi}\right)_{(a_{1},\, a/d)}\left[\begin{pmatrix}
		1 &            &         \\
		& \alpha & \beta \\
		& -a_{0} & d
	\end{pmatrix}g\right],
\end{align}
where  $(\alpha, \beta)$ is any pair of integers satisfying 
\begin{align}\label{extrstr}
	d\alpha+a_{0}\beta \ = \ 1. 
\end{align} 
 \end{prop}

 \begin{proof}
 Firstly, we perform a Fourier expansion  along the abelian unipotent  subgroup
 	$\left\{\begin{psmallmatrix}
 		1 &    & * \\
 		& 1 &   \\
 		&    & 1
 	\end{psmallmatrix}\right\}$, which leads to
 	\begin{align}
 	\hspace{20pt}	\int_{0}^{1}  \ \Phi\left[
 		\begin{pmatrix}
 			1 & u_{1,2} &  \\
 			& 1 & \\
 			&    & 1
 		\end{pmatrix} g\right]\,  & e(-au_{1,2}) \ du_{1,2} \nonumber\\
 		\ = \ &  \sum_{a_{0}=-\infty}^{\infty} \ 	\int_{0}^{1}  \ 	\int_{0}^{1}   \ \Phi\left[
 		\begin{pmatrix}
 			1 & u_{1,2} & u_{1,3}  \\
 			& 1          & \\
 			&             & 1
 		\end{pmatrix} g\right] e(-au_{1,2}-a_{0} u_{1,3}) \ du_{1,2} \ du_{1,3}. 
 	\end{align}
 	
 Secondly,  we consider a change of variables of the form $(u_{1,2}, u_{1,3}) \ = \  (u_{1,2}', u_{1,3}') \cdot \begin{psmallmatrix}
 		\alpha     & \beta \\
 		\gamma &\delta
 	\end{psmallmatrix}$ for each $a_{0}\in \Z$, where the matrix $\begin{psmallmatrix}
 		\alpha     & \beta \\
 		\gamma &\delta
 	\end{psmallmatrix} $ satisfies
  \begin{itemize}
 	\item  $	\gamma \ = \ \gamma(a_{0}) \ := \  -a_{0}/ \gcd\left(a_{0}, a\right)$ \ and \ $\delta \ = \ \delta(a_{0}) \ := \ a/\gcd\left(a_{0}, a\right)$;
 	
 	\item  $\alpha=\alpha(a_{0}), \ \beta=\beta(a_{0})$ be any pair of integers that satisfy \ $	\delta \alpha- \gamma\beta \ = \ 1$. 
 \end{itemize}
  Then $au_{1,2}+a_{0} u_{1,3}  =   a'u_{1,2}' $  \ with  \  $a'   =  a'(a_{0})  :=  a\alpha+a_{0}\beta$, and one can easily verify that 
  \begin{align*}
  	\hspace{20pt} 	\begin{pmatrix}
  		1 & u_{1,2} & u_{1,3}  \\
  		& 1          & \\
  		&             & 1
  	\end{pmatrix} 
  	\ = \  
  	\begin{pmatrix}
  		1 &                     &\\
  		&  \delta          &   -\beta\\
  		&   -\gamma   & \alpha
  	\end{pmatrix}
  	\begin{pmatrix}
  		1 & u_{1,2}' & u_{1,3}'  \\
  		& 1           & \\
  		&              & 1
  	\end{pmatrix} 
  	\begin{pmatrix}
  		1 &            &         \\
  		& \alpha & \beta \\
  		& \gamma & \delta 
  	\end{pmatrix}.
  \end{align*}
 The automorphy of $\Phi$ with respect to  $\Gamma_{3}$ gives 
 	\begin{align}
 	\hspace{20pt}	\int_{0}^{1} \ \Phi\left[
 		\begin{pmatrix}
 			1 & u_{1,2} &  \\
 			& 1          & \\
 			&             & 1
 		\end{pmatrix} 
 		g\right] \, & e(- a u_{1,2})\ du_{1,2} \nonumber\\
 		&\hspace{-40pt} \ = \  \sum_{a_{0}=-\infty}^{\infty} \  	\int_{0}^{1} \ 	\int_{0}^{1} \  \Phi\left[
 		\begin{pmatrix}
 			1 & u_{1,2}' & u_{1,3}'  \\
 			& 1           &                \\
 			&              & 1
 		\end{pmatrix} 
 		\begin{pmatrix}
 			1 &            &           \\
 			& \alpha & \beta    \\
 			&\gamma & \delta
 		\end{pmatrix}
 		g\right] e(- a'u_{1,2}') \ du_{1,2}' \ du_{1,3}'. 
 	\end{align}
 	
 	Thirdly, we perform a Fourier expansion along another abelian unipotent subgroup
 	$\left\{ \begin{psmallmatrix}
 		1 &     &    \\
 		& 1  & *  \\
 		&     & 1
 	\end{psmallmatrix}\right\}$. From this, we obtain 	\vspace{-0.2in} 
 	\begin{align}
 \hspace{25pt} 	\int_{0}^{1} \ \Phi\left[
 	\begin{pmatrix}
 		1 & u_{1,2} &  \\
 		& 1          & \\
 		&             & 1
 	\end{pmatrix} g\right] \, e( -a u_{1,2}) \ du_{1,2} 
 	\ \  =  \ \  
 	\sum_{a_{0},a_{1}=-\infty}^{\infty} \ \left(\widehat{\Phi}\right)_{(a_{1}, a')}\left[\begin{pmatrix}
 		1 &            &         \\
 		& \alpha & \beta \\
 		& \gamma  & \delta
 	\end{pmatrix}
 	g\right]. 
 \end{align}
  Upon breaking up the $a_{0}$-sum according to the value of $d:=(a_{0}, a)$ and noticing that $a'=d$, we have
   \vspace{-0.2in} 
  	\begin{align}
\hspace{10pt}	\int_{0}^{1} \ \Phi\left[
  	\begin{pmatrix}
  		1 & u_{1,2} &  \\
  		& 1          & \\
  		&             & 1
  	\end{pmatrix} g\right]\,   e( -a u_{1,2}) \ du_{1,2}  
  	\  =   \   \sum_{d\mid a} \ \sum_{\substack{a_{0}\in \Z \\ (a_{0},a )=d}} \ 
  	\sum_{a_{1}\in \Z} \ \left(\widehat{\Phi}\right)_{(a_{1}, d)}\left[\begin{pmatrix}
  		1 &            &         \\
  		& \alpha & \beta \\
  		& -a_{0}/d & a/d
  	\end{pmatrix}
  	g\right],
  \end{align}
  where $(\alpha, \beta)$ is any pair of integers satisfying $a\alpha+ a_{0}\beta =d$.  The conclusion (\ref{keylemform}) now follows from  the successive replacements $a_{0} \to da_{0}$ and   $d \to a/d$.
 
 \end{proof}


 The next step concerns the explication of  the pair  $(\alpha, \beta)$ in (\ref{incomf}) in which condition (\ref{extrstr}) will naturally come into play. To some extent, this may be compared  to  the step of applying  `additive reciprocity'   in  \cite{BHKM20}, where  its usage (and actually the full compositum of the transformations) was guided by the lengths of summations instead, see the sketch   \footnote{ Justifying (or  motivating) the insertion of various auxiliary additive factors in the main argument of \cite{BHKM20} in which the method of continuation was used.   }  in Section 1.4 therein.

 \begin{cor}\label{incomexpf}
 	Suppose $\Phi: \mathfrak{h}^3\rightarrow \C$ is an automorphic  form of $\Gamma_{3}$. Then for any $y_{0}, y_{1}>0$, we have
 	\begin{align}\label{incomexpform}
	\hspace{10pt} 	\int_{0}^{1} \ \Phi\left[
	\begin{pmatrix}
		1 & u_{1,2} &  \\
		& 1          & \\
		&             & 1
	\end{pmatrix} 
	\begin{pmatrix}
		y_{0}y_{1}&  & \\
		& y_{0} & \\
		&           & 1
	\end{pmatrix}\right]  &  \ e( -a u_{1,2}) \ du_{1,2}  \ \nonumber\\[-0.1in]
	&\nonumber\\ 
	\ &\hspace{-170pt} = \ \   \sum_{d \mid a} \ 
	\ \sum_{a_{1}\in \Z}  \  \left(\widehat{\Phi}\right)_{(a_{1},\, a/d)}\begin{pmatrix}
		y_{0}y_{1}&  & \\
		& y_{0} & \\
		&           & 1
	\end{pmatrix}   \nonumber\\[0.1in]
  &\hspace{-140pt}\ + \   \sum_{d \mid a} \ 
\sum_{\substack{a_{0} \in \Z -\{0\} \\ \gcd(a_{0}, d)=1}} \ 	\left(\widehat{\Phi}\right)_{(0,\, a/d)}\begin{pmatrix}
	\frac{y_{0} y_{1}}{\sqrt{(a_{0}y_{0})^2+d^2}} &  &  \\
	& \frac{y_{0}}{(a_{0} y_{0})^2+d^2} &  \\
	&  & 1 \\
\end{pmatrix}  \nonumber\\[0.1in]
	&  \hspace{-110pt}  \ + \   \sum_{d \mid a} \ 
	\sum_{\substack{a_{0} \in \Z -\{0\} \\ \gcd(a_{0}, d)=1}} \ \sum_{a_{1} \in \Z-\{0\}} \  e\left(\frac{a_{1}\overline{a_{0}}}{d}\right)\  \nonumber\\
	&\hspace{-80pt}  \cdot \left(\widehat{\Phi}\right)_{(a_{1},\, a/d)}\left[  \begin{pmatrix}
		1 &     &    \\
		& 1  & -\frac{1}{da_{0}}\frac{(a_{0}y_{0})^2}{d^2+(a_{0}y_{0})^2} \\
		&     &    1
	\end{pmatrix} \begin{pmatrix}
		\frac{y_{0} y_{1}}{\sqrt{(a_{0}y_{0})^2+d^2}} &  &  \\
		& \frac{y_{0}}{(a_{0} y_{0})^2+d^2} &  \\
		&  & 1 \\
	\end{pmatrix}\right], \nonumber\\
	\nonumber\\
\end{align}
where $a_{0}\overline{a_{0}}\equiv 1 \  (\bmod\, d)$ for\,  $\gcd(a_{0},d)=1$. 

 \end{cor}

\begin{rem}
	\ 
	\begin{enumerate}
		
		\item Corollary \ref{incomexpf} can be understood as follows:  the exponential factor  $ e\left(a_{1}\overline{a_{0}}/d\right)$ in  (\ref{incomexpform}) is our desired `arithmetic twist', whereas the last line of (\ref{incomexpform}) will be regarded as the `analytic component'.  \vspace{0.1in}

		\item   The mentioned `analytic component' is more-or-less  the same as the one  for  the untwisted case of \cite{Kw23, Kw23a+}. This is anticipated and is a reflection of the local-global nature of our method.  
		\vspace{0.1in} 
		
		\item The expression (\ref{incomexpform}) will be simplified considerably in (\ref{cleanup}) upon plugging into the period integral $	\left(P^a, \, \mathbb{P}_{2}^{3} \Phi\cdot  |\det *|^{\overline{s}-\frac{1}{2}}\right)_{\Gamma_{2}\setminus GL_{2}(\R)}$. 
		\vspace{0.1in}

	\item  Corollary \ref{incomexpf} was stated  in a  slightly more general form as  it has further  applications in verifying Conjecture \ref{twiscubconj}, i.e., the twisted version of the main theorem of \cite{Kw23a+}, which requires us to take $\Phi= (E_{\min})^{^{*}}(\, *\, ; \, (\alpha_{1}, \alpha_{2}, \alpha_{3}))$.  When $\Phi$ is cuspidal, we have     $(\widehat{\Phi})_{(0, \, a/d)} \equiv 0$, so the second summand on the right side of  (\ref{incomexpform}) vanishes. 
	
	\end{enumerate}
	
\end{rem}

 \begin{proof}[Proof of Corollary \ref{incomexpf}]
 When $a_{0}=0$,  we have $d= \alpha=1$ and we may choose $\beta=0$ in Proposition \ref{incomf}. This results in the first term on the right side of (\ref{incomexpform}). 
 
 Suppose now $a_{0}\neq 0$\,  (with\,  $\gcd(a_{0},d)=1$).  Let   $w_{\ell}:= \begin{psmallmatrix}
 	&      & 1\\
 	& 1   &   \\
 	1	 &      & 
 \end{psmallmatrix}$ \ and \ $g^{\iota}\ := \ ^{t}g^{-1}$.  To reveal the non-archimedean component behind our calculations,  we must deviate from \cite{Kw23} by first  considering the product of matrices:
 \begin{align*}
 	\hspace{10pt} w_{\ell} \ \left[ 
 	\begin{pmatrix}
 		1 &            &         \\
 		& \alpha & \beta \\
 		& -a_{0} & d
 	\end{pmatrix}
 	\begin{pmatrix}
 		y_{0}y_{1}&  & \\
 		& y_{0} & \\
 		&           & 1
 	\end{pmatrix}  \right]^{\iota} w_{\ell},
 \end{align*}
which can be seen to be  
\begin{align*}
\hspace{70pt} 	\begin{pmatrix}
		y_{1}
		\begin{pmatrix}
			\alpha & -\beta \\
			a_{0}   & d
		\end{pmatrix}
		\begin{pmatrix}
			y_{0} & \\
			& 1
		\end{pmatrix}
		& \\
		& 1
	\end{pmatrix} \hspace{15pt}   (\bmod\  \R^{\times}).
\end{align*}
Thus, we are  left with a simple $2$-by-$2$ Iwasawa computation: 
\begin{align}\label{2by2Iwa}
\hspace{10pt} \begin{pmatrix}
	\alpha & -\beta \\
	a_{0}   & d
\end{pmatrix}
\begin{pmatrix}
	y_{0} & \\
	& 1
\end{pmatrix} \ \equiv \  \sqrt{d^{2}+(a_{0}y_{0})^2} \  \begin{pmatrix}
1 & \frac{\alpha}{a_{0}}- \frac{d/a_{0}}{d^2+(a_{0}y_{0})^2} \\
  & 1
\end{pmatrix}
\begin{pmatrix}
	\frac{y_{0}}{d^2+(a_{0}y_{0})^2} & \\
	 & 1
\end{pmatrix} \hspace{10pt} (\bmod\, SO(2)). 
\end{align}
 From (\ref{2by2Iwa}), the Iwasawa 	decomposition of\,  $\begin{psmallmatrix}
 	1 &            &         \\
 	& \alpha & \beta \\
 	& -a_{0} & d
 \end{psmallmatrix}
 \begin{psmallmatrix}
 	y_{0}y_{1}&  & \\
 	& y_{0} & \\
 	&           & 1
 \end{psmallmatrix} \ (\bmod\, O_{3}(\R) \, \R^{\times})$ can be readily obtained and it follows that 
\begin{align}\label{centralmaid}
\hspace{5pt} 	&	\left(\widehat{\Phi}\right)_{(a_{1}, \, a/d)}  \left[\begin{pmatrix}
		1 &            &         \\
		& \alpha & \beta \\
		& -a_{0} & d
	\end{pmatrix}
	\begin{pmatrix}
		y_{0}y_{1}&  & \\
		& y_{0} & \\
		&           & 1
	\end{pmatrix}\right] \nonumber\\[0.1in] 
	& \hspace{70pt}  \ \ = \  \  	\left(\widehat{\Phi}\right)_{(a_{1},\,  a/d)}\left[ \begin{pmatrix}
	1 &     &    \\
	& 1  & \frac{d/a_{0}}{d^2+(a_{0}y_{0})^2} - \frac{\alpha}{a_{0}}\\
	&     &    1
	\end{pmatrix} \begin{pmatrix}
	\frac{y_{0} y_{1}}{\sqrt{(a_{0}y_{0})^2+d^2}} &  &  \\
	& \frac{y_{0}}{(a_{0} y_{0})^2+d^2} &  \\
	&  & 1 \\
	\end{pmatrix}\right]. \nonumber\\
\end{align}

When $a_{1}=0$, we have 
\begin{align}
	 \hspace{5pt}	\left(\widehat{\Phi}\right)_{(a_{1},\,  a/d)}  \left[\begin{pmatrix}
		1 &            &         \\
		& \alpha & \beta \\
		& -a_{0} & d
	\end{pmatrix}
	\begin{pmatrix}
		y_{0}y_{1}&  & \\
		& y_{0} & \\
		&           & 1
	\end{pmatrix}\right] \ \ = \ \    &  	\left(\widehat{\Phi}\right)_{(0,\,  a/d)}\begin{pmatrix}
	\frac{y_{0} y_{1}}{\sqrt{(a_{0}y_{0})^2+d^2}} &  &  \\
	& \frac{y_{0}}{(a_{0} y_{0})^2+d^2} &  \\
	&  & 1 \\
\end{pmatrix} \nonumber\\
\end{align}
from  (\ref{centralmaid}) and   (\ref{unipotran}).  This  results in the second term on the right side of (\ref{incomexpform}). When  $a_{1}\in \Z-\{0\}$,  observe from (\ref{extrstr}) that 
\begin{align*}
	\frac{d/a_{0}}{d^2+(a_{0}y_{0})^2} \ - \ \frac{\alpha}{a_{0}} \ \equiv \    \frac{\overline{a_{0}}}{d} \ + \ 	\frac{d/a_{0}}{d^2+(a_{0}y_{0})^2}   -\frac{1}{a_{0}d} \ = \   \frac{\overline{a_{0}}}{d}  \ - \ \frac{1}{da_{0}}\, \frac{(a_{0}y_{0})^2}{d^2+(a_{0}y_{0})^2} \ \ (\bmod\, 1),
\end{align*}
and  (\ref{centralmaid}) can be rewritten as 
\begin{align*}
		\left(\widehat{\Phi}\right)_{(a_{1},\,  a/d)}\left[ \begin{pmatrix}
		1 &     &    \\
		& 1  &  \overline{a_{0}}/d\\
		&     &    1
	\end{pmatrix} 
	\, \cdot \,  
	\begin{pmatrix}
		1 &     &    \\
		& 1  &    -  \frac{1}{da_{0}}\, \frac{(a_{0}y_{0})^2}{d^2+(a_{0}y_{0})^2} \\ 
		&     &    1
	\end{pmatrix} 
	\begin{pmatrix}
		\frac{y_{0} y_{1}}{\sqrt{(a_{0}y_{0})^2+d^2}} &  &  \\
		& \frac{y_{0}}{(a_{0} y_{0})^2+d^2} &  \\
		&  & 1 \\
	\end{pmatrix}\right]. 
\end{align*}
We thus obtain the third term on the right side of (\ref{incomexpform}) by invoking  (\ref{unipotran}).   This concludes the proof of Corollary  \ref{incomexpf}. 

 \end{proof}


\section{Proof of Theorem \ref{twisdirc}}\label{twist}

 	\subsection{Unfolding and Clean-up}\label{inisect}
 	Let $\epsilon:= 1/100$.  On the vertical strip $  1+\theta/2 +\epsilon< \sigma < 4$, we begin by replacing the Poincar\'{e} series $P^{a}$ by its definition in the pairing
 \begin{align}\label{RSpair}
 	 2a^{-1/2}  \, \left(P^a, \ \mathbb{P}_{2}^{3} \Phi\cdot  |\det *|^{\overline{s}-\frac{1}{2}}\right)_{\Gamma_{2}\setminus GL_{2}(\R)}. 
 \end{align}
An unfolding shows that (\ref{RSpair}) is equal to 
  \begin{align}
	 2a^{-1/2}\, \int_{0}^{\infty}  \int_{0}^{\infty} \  h(ay_{1})\, (y_{0}^2y_{1})^{s-\frac{1}{2}} \,   \int_{0}^{1} \hspace{5pt}   \widetilde{\Phi}\left[
	\begin{pmatrix}
		1 & u_{1,2} &  \\
		& 1          & \\
		&             & 1
	\end{pmatrix} 
	\begin{pmatrix}
		y_{0}y_{1} &             &  \\
		& y_{0}    &   \\
		&             & 1
	\end{pmatrix}\right] e(au_{1,2}) \ du_{1,2} \ \frac{dy_{0}\, dy_{1}}{y_{0}y_{1}^2}. \label{alunf}
 \end{align}
From Corollary \ref{incomexpf} (replacing $\Phi$ therein by $\widetilde{\Phi}$ and noticing that   $(\widehat{\Phi})_{(0,a/d)} \equiv 0$ as $\Phi$ is a cusp form here), we obtain
\begin{align}\label{cleantwistwey}
	 2a^{-1/2}  \,  &\sum_{d \mid a} \ 
		\ \sum_{a_{1}\in \Z-\{0\}}  \    \int_{0}^{\infty} \int_{0}^{\infty} h(ay_{1}) \left(\widehat{\widetilde{\Phi}}\right)_{(a_{1}, a/d)} \begin{pmatrix}
			y_{0}y_{1}&  & \\
			& y_{0} & \\
			&           & 1
		\end{pmatrix} (y_{0}^2 y_{1})^{s-\frac{1}{2}} \ \frac{dy_{0} \, dy_{1}}{y_{0} y_{1}^2} \nonumber\\[0.1in]
&   \ + \  2a^{-1/2} \,  \sum_{d \mid a} \ 
\sum_{\substack{a_{0} \in \Z -\{0\} \\ \gcd(a_{0}, d)=1}} \ \sum_{a_{1} \in \Z-\{0\}} \	   \ \frac{\mathcal{B}_{\Phi}\left(a/d, a_{1}   \right)}{|a_{0}|^{2s-1} |a_{1}|}\ e\left(-\frac{a_{1}\overline{a_{0}}}{d}\right) \nonumber\\[0.1in] 
&\hspace{50pt} \cdot  \,  \int_{0}^{\infty} \int_{0}^{\infty} h(ay_{1}) \, W_{-\alpha(\Phi)}\left( \frac{|a_{1}|y_{0}}{(a_{0} y_{0})^2+d^2}, \ \frac{ay_{1}}{d}\sqrt{(a_{0} y_{0})^2+d^2} \right) \nonumber\\[0.1in] 
&\hspace{200pt} \cdot \, e\left(\frac{1}{da_{0}}\frac{(a_{0}y_{0})^2}{d^2+(a_{0}y_{0})^2}\right)  (y_{0}^2 y_{1})^{s-\frac{1}{2}}  \ \ \frac{dy_{0}\, dy_{1}}{y_{0} y_{1}^2} 
\end{align}
and the absolute convergence of (\ref{cleantwistwey}) follows from Proposition 6.5 of \cite{Kw23}.

 Denote by $D_{\Phi}^{(a)}(s)$ the expression on the first line of (\ref{cleantwistwey})  and  $OD_{\Phi}^{(a)}(s)$ the expression spanning the last three lines of (\ref{cleantwistwey}). In other words,
 \begin{align}\label{diagoffsplt}
 	2a^{-1/2}  \, \left(P^a, \ \mathbb{P}_{2}^{3} \Phi\cdot  |\det *|^{\overline{s}-\frac{1}{2}}\right)_{\Gamma_{2}\setminus GL_{2}(\R)}  \ \ =  \ \  D_{\Phi}^{(a)}(s) \ + \   OD_{\Phi}^{(a)}(s).
 \end{align}
 By the changes of variables $y_{0}\to |a_{1}|^{-1} y_{0}$ and $y_{1}\to a^{-1} y_{1}$, we have
 \begin{align}\label{twisdiag}
 	D_{\Phi}^{(a)}(s) \ = \ 2a^{-s}	\sum_{a_{1}\in \Z-\{0\}} \frac{\mathcal{B}_{\Phi}(a, a_{1})}{|a_{1} |^{2s}}  \cdot  \int_{0}^{\infty} \int_{0}^{\infty} h(y_{1})\,    (y_{0}^2 y_{1})^{s-\frac{1}{2}} \, W_{-\alpha(\Phi)}\left(y_{0},  y_{1}\right) \,   \frac{dy_{0}\, dy_{1}}{y_{0}y_{1}^2}, 
 \end{align}
where the double integral can be computed  by \eqref{eqn sta1} and (\ref{invers}). For $OD_{\Phi}^{(a)}(s)$, we apply another set of  changes of variables $y_{0}\to (d|a_{0}|^{-1}) y_{0}$ and $y_{1}\to a^{-1} y_{1}$, it follows at once 
\begin{align}\label{cleanup}
	OD_{\Phi}^{(a)}(s) \ \ = \ \  &  2a^{-s} \  \sum_{d \mid a} \  d^{2s}
	\sum_{\substack{a_{0} \in \Z-\{0\}  \\ \gcd(a_{0}, d)=1}} \  \sum_{a_{1} \in \Z-\{0\}}  \ \frac{\mathcal{B}_{\Phi}\left(a/d, a_{1}   \right)}{|a_{0}|^{2s-1} |a_{1}|}\,  e\left(-\frac{a_{1}\ovA{a_{0}}}{d}\right) \nonumber\\
	&\hspace{90pt}\cdot \ \int_{0}^{\infty} \int_{0}^{\infty} h(y_{1})\,  (y_{0}^2 y_{1})^{s-\frac{1}{2}}  \, e\left( \frac{a_{1}}{da_{0}}\, \frac{y_{0}^2}{1+y_{0}^2}\right)   \ \nonumber\\[0.1in]
	& \hspace{165pt} \cdot W_{-\alpha(\Phi)} \left(\left|\frac{a_{1}}{da_{0}}\right|\, \frac{ y_{0}}{ 1+y_{0}^2}, \ y_{1}\sqrt{1+y_{0}^2}   \right)   \  \frac{dy_{0} \, dy_{1}}{y_{0}y_{1}^2}. \nonumber\\
\end{align}

Following the same argument presented in Proposition 7.2 of  \cite{Kw23}, i.e.,  applying the Mellin inversion formulae for $e(\cdots)$ and $W_{-\alpha(\Phi)}(\cdots)$ for separation of variables, we obtain
\begin{align}\label{prop72kw23}
	OD_{\Phi}^{(a)}(s;\,  \phi) \ \ = \ \    2a^{-s} \cdot \frac{1}{8} \cdot \,  \int_{(1+\theta+2\epsilon)} \   \sum_{\pm} \ & \sum_{d \mid a} \  d^{2s+s_{0}-1}
	 \ \sum_{\substack{a_{0}, a_{1} \in \Z-\{0\}  \\ \gcd(a_{0}, d)=1 \\ \sgn(a_{0}a_{1})= \pm }} \    \ \frac{\mathcal{B}_{\Phi}\left(a/d, a_{1}   \right)}{|a_{0}|^{2s-s_{0}} |a_{1}|^{s_{0}}}\,  e\left(-\frac{a_{1}\ovA{a_{0}}}{d}\right) \nonumber\\
	 &\hspace{30pt} \cdot\,   (\mathcal{F}_{\Phi}^{(\pm)}H)\left(s_{0}, \, s; \,  \phi\right) \ \frac{ds_{0}}{2\pi i},
\end{align}
where the integral transform $\mathcal{F}_{\Phi}^{(\pm)}H$ was  defined in Section  \ref{Stirl}, and 
\begin{align}
		\lim_{\phi \to \pi/2 } \ OD_{\Phi}^{(a)}(s;\, \phi) \ \ = \  \ OD_{\Phi}^{(a)}(s)
\end{align}
holds on the domain $(3+\theta)/2< \sigma < 4$ (see Proposition 7.3 of \cite{Kw23}). For  the readers' convenience, more details will be indicated  in Section \ref{suppInt}.  Note: the constant $1/8$ in (\ref{prop72kw23}) comes from the constants $1/4$ and $1/2$ from Vinogradov-Takhtadzhyan's formula and  Euler's beta integral formula respectively. 

However, one must beware of  the signs of $a_{0}a_{1}$ due to the extra factor $e(-a_{1}\overline{a_{0}}/d)$ present in this article (but not in  \cite{Kw23, Kw23a+}), which will in turn influence the shapes of the dual moments, see  (\ref{expinmul}) and  (\ref{eisdualmo}). Consider the double Dirichlet series
\begin{align}\label{specLser}
	\mathcal{L}^{(a)}_{\pm}\left(s_{0}, s; \Phi \right) \ \ &:= \  \  \sum_{d \mid a} \  d^{2s+s_{0}-1}
	\sum_{\substack{a_{0} \ge 1 \\ \gcd(a_{0}, d)=1}}  \ \sum_{a_{1}\ge 1 }  \ \frac{\mathcal{B}_{\Phi}\left( a/d, a_{1}   \right)}{(a_{0})^{2s-s_{0}} (a_{1})^{s_{0}}}\  e\left(\mp \frac{a_{1}\ovA{a_{0}}}{d}\right),
\end{align}
which converges absolutely on the region  
\begin{align}\label{regAC}
	\re (2s-s_{0}) \ > \ 1 \hspace{20pt} \text{ and  } \hspace{20pt}\re s_{0} \ > \ 1+\theta. 
\end{align}
 Using the fact that $\Phi$ is an even form, we readily observe  that
\begin{align}
	OD_{\Phi}^{(a)}(s;\, \phi) \ \ = \ \   &\frac{a^{-s}}{2} \  \sum_{\pm} \  \int_{(1+\theta+2\epsilon)} \  	\mathcal{L}^{(a)}_{\pm}\left(s_{0}, s; \Phi \right)   (\mathcal{F}_{\Phi}^{(\pm)}H)\left(s_{0}, \, s; \,  \phi\right) \ \frac{ds_{0}}{2\pi i}. 
\end{align}
for $ 1+ \theta/2  +\epsilon< \sigma < 4$  and   $\phi\in (0, \pi/2)$.


 \subsection{Analytic Continuation}
We will make use of Proposition \ref{simpvor} to obtain  the analytic properties of  $\mathcal{L}_{\pm}^{(a)}\left(s_{0}, s; \Phi\right)$.

\begin{prop}\label{anacondtwi}
	The double Dirichlet series $\mathcal{L}_{\pm}^{(a)}(s_{0},s; \Phi)$ admits a holomorphic   continuation to $\C^2$  except on the polar divisor $2s-s_{0}=1$.  
\end{prop}

\begin{proof}
Suppose  $\re (2s-s_{0})>1$ and $\re s_{0}> 1+\theta$. The double Dirichlet series (\ref{specLser}) converges  absolutely and can be written as
\begin{align}
	\mathcal{L}_{\pm}^{(a)}\left(s_{0}, s ; \Phi\right) \ &= \  
	\ \sum_{d \mid a} \  d^{2s+s_{0}-1} \ \sideset{}{^*}{\sum}_{\ell \ (\bmod\, d)} \ 
	\sum_{\substack{a_{0} \ge 1 \\ a_{0} \ \equiv \ \ell \ (\bmod\, d)}}  \ \sum_{a_{1}\ge 1}  \ \frac{\mathcal{B}_{\Phi}\left( a/d, a_{1}   \right)}{(a_{0})^{2s-s_{0}} (a_{1})^{s_{0}}} \   e\left( \mp \frac{a_{1}\ovA{\ell}}{d}\right)
\end{align}
upon splitting the $a_{0}$-sum into  residue classes $(\bmod\, d)$.  Since
\begin{align}\label{congHurw}
	\sum_{\substack{a_{0} \ge 1 \\ a_{0} \ \equiv \  \ell \ (\bmod\, d)}}  \frac{1}{a_{0}^{2s-s_{0}}} \ &= \  d^{-(2s-s_{0})}\cdot  \zeta\left( 2s-s_{0}, \frac{\ell}{d}\right),
\end{align}
we have
\begin{align}\label{prevorofor}
	\mathcal{L}_{\pm}^{(a)}\left(s_{0}, s; \Phi \right) \ \ = \ \   \sum_{d \mid a} \  d^{2s_{0}-1} \ \sideset{}{^*}{\sum}_{\ell \ (\bmod\, d)} \  \zeta\left( 2s-s_{0}, \frac{\ell}{d}\right)  \left( \sum_{a_{1}=1}^{\infty}  \ \frac{\mathcal{B}_{\Phi}\left( a/d,  a_{1}   \right)}{ a_{1}^{s_{0}}}\  e\left(\mp \frac{a_{1}\ovA{\ell}}{d}\right) \right). 
\end{align}
The Hecke relation  (\ref{splGL3Hec}) implies that
 \begin{align}\label{twistDS}
 	\mathcal{L}^{(a)}_{\pm}\left(s_{0}, s; \Phi \right)  \ \ = \ \    \sum_{dr \mid a} \   \frac{d^{2s_{0}-1} \mu(r)}{ r^{s_{0}} }\ \overline{\lambda_{\Phi}}\left( \frac{a}{dr}\right)  \  \sideset{}{^*}{\sum}_{\ell \ (\bmod\, d)} \  \zeta\left( 2s-s_{0}, \frac{\ell}{d}\right) 	L\left( s_{0}; \mp \frac{r\overline{\ell}}{d}; \Phi\right)
 \end{align}
and its analytic continuation now follows  from Proposition  \ref{simpvor}. Also, 
\begin{align}\label{restwisDS}
- \ 	\Res_{s_{0}=2s-1} \  \mathcal{L}_{\pm}^{(a)}(s_{0}, s; \Phi) \ \  =  \ \  \sum_{dr \mid a} \   \frac{d^{4s-3} \mu(r)}{ r^{2s-1} }\ \overline{\lambda_{\Phi}}\left( \frac{a}{dr}\right)  \ \  \sideset{}{^*}{\sum}_{\ell \ (\bmod\, d)} \ 	L\left(2s-1; \mp \frac{r\overline{\ell}}{d}; \Phi\right)
 \end{align}
which  clearly admits an entire continuation.  This finishes the proof. 
\end{proof}

Denote by $\mathcal{R}_{\pm}^{(a)}(s; \Phi)$ the expression on the right side of (\ref{restwisDS}).   By Corollary \ref{anacondtwi} and repeating the argument described in Section 9 of \cite{Kw23} (with the exact same  analysis at the archimedean place, see Section \ref{Stirl}), we arrive at
\begin{prop}\label{repaCarg}
	On $1/4+\epsilon/2< \sigma < 3/4$, we have
	\begin{align}
	\hspace{30pt} 	OD_{\Phi}^{(a)}(s) \ \ = \ \   &\frac{a^{-s}}{2} \  \sum_{\pm} \  \int_{(1/2)} \  	\mathcal{L}^{(a)}_{\pm}\left(s_{0}, s; \Phi \right)   (\mathcal{F}_{\Phi}^{(\pm)}H)\left(s_{0}, \, s\right) \ \frac{ds_{0}}{2\pi i} \nonumber\\
		&\hspace{120pt} \ + \  \frac{a^{-s}}{2}  \ \sum_{\pm} \  \  \mathcal{R}_{\pm}^{(a)}(s; \Phi) (\mathcal{F}_{\Phi}^{(\pm)} H)\left(2s-1,\,   s\right). 
	\end{align}
\end{prop}


\subsection{The Diagonal Main Term}
 	 
 	 \begin{lem}\label{hecdoub}
 	 	For $\re s \gg 1$, we have
 	 	\begin{align}\label{twisdiaDS}
 	 		\frac{1}{2} \ \sum_{a_{1} \in \Z-\{0\}} \ \frac{\mathcal{B}_{\Phi}(a, a_{1})}{|a_{1} |^{2s}} \ \ &= \   L(2s, \Phi)  \  \sum_{r\mid a}  \ \frac{\mu(r) \overline{\lambda_{\Phi}}(a/r)}{r^{2s}} \nonumber\\
 	 		\  \ &= \  \  L(2s, \Phi)  \ \prod_{p\mid a} \ \left\{   \overline{\lambda_{\Phi}}\left(p^{o_{p}(a)}\right) \ - \   \frac{\overline{\lambda_{\Phi}}\left(p^{o_{p}(a)-1}\right) }{p^{2s}} \right\}, 
 	 	\end{align}
  	where $o_{p}(a)$ denotes the power of $p$ in the prime factorization of $a$.  Thus,  the  Dirichlet series above  admits an entire continuation.

 	 \end{lem}

 	 \begin{proof}
 	 	The first equality follows from (\ref{splGL3Hec}) and the second equality follows from the multiplicativity  of $\lambda_{\Phi}(\,\cdot\, )$, $\mu(\,\cdot\, )$, and their  multiplicative convolution. 
 	 \end{proof}

 	\begin{prop}\label{maintwisdiag}
 		On the domain  $1+ \theta/2 +\epsilon< \sigma < 4$, we have
 		 \begin{align}\label{twisdiaMT}
 			D_{\Phi}^{(a)}(s) \ = \ \frac{ a^{-s}}{2}  \ L(2s, \Phi) \   \prod_{p\mid a} \ & \left\{  \overline{\lambda_{\Phi}}\left(p^{o_{p}(a)}\right) \ - \   \frac{\overline{\lambda_{\Phi}}\left(p^{o_{p}(a)-1}\right) }{p^{2s}} \right\}	\nonumber\\
 		&	\hspace{40pt}\  \cdot \   \int_{(0)}  \frac{H(\mu)}{|\Gamma(\mu)|^2} \cdot \prod_{\pm} \ \prod_{i=1}^{3} \  \Gamma_{\R}\left(s\pm \mu-\alpha_{i}\right)\ \frac{d\mu}{2\pi i}. 
 		\end{align}
 	As a result,  $	D_{\Phi}^{(a)}(s) $  admits a holomorphic continuation to $0< \sigma < 4$. 
 	\end{prop}
 
 \begin{proof}
 	Follows from (\ref{twisdiag}),  (\ref{invers}), \eqref{eqn sta1} and Lemma \ref{hecdoub}. 
 \end{proof}


 	 \subsection{The Off-diagonal Main Term}\label{ODMt}
If $\Phi$ is a Maass cusp form of $SL_{3}(\Z)$, the  twisted recipe of \cite{CFKRS05} for the spectral  moment of  (\ref{dirimom})  is implied by the functional equation 
 	 \begin{align}
 	 	\Lambda(s, \, \phi\otimes \widetilde{\Phi}) \ = \   	\Lambda(1-s,\, \phi\otimes \Phi)
 	 \end{align}
 with  $\phi$ being an even Maass form of $SL_{2}(\Z)$, see Theorem 12.2.5 of \cite{Go15}.  Indeed,   the spectral side of (\ref{dirimom})  is  invariant under the replacements  $s \to 1-s$ and $\Phi \rightarrow \widetilde{\Phi}$. This observation together with the diagonal evaluation  (\ref{twisdiaMT}) predicts the existence of an off-diagonal  main term:
 	 \begin{align}\label{dualtwis}
 	 \hspace{15pt} 	\frac{a^{s-1}}{2}  \,    L\left( 2\, (1-s), \widetilde{\Phi}\right) & \,   \prod_{p\mid a} \ \left\{   \lambda_{\Phi}\left(p^{o_{p}(a)}\right) \ - \   \frac{\lambda_{\Phi}\left(p^{o_{p}(a)-1}\right) }{p^{2(1-s)}} \right\} \nonumber\\[0.1in]
 	 	& \hspace{60pt} \cdot \,  \int_{(0)} \ \frac{H(\mu)}{\left|\Gamma(\mu)\right|^2} \cdot \prod_{\pm} \ \prod_{i=1}^{3} \ \Gamma_{\R}\left(1-s \pm \mu + \alpha_{i}\right) \ \frac{d\mu}{2\pi i}. 
 	 \end{align}
 	 In other words, we must show that
  \begin{prop}\label{twismatchporp}
For $1/4 +\epsilon/2< \sigma < 3/4$, the expression
  \begin{align}\label{matchagainstwsi}
  	\frac{a^{-s}}{2} \   \sum_{\pm} \     \mathcal{R}_{\pm}^{(a)}(s; \Phi) \, (\mathcal{F}_{\Phi}^{(\pm)} H)\left(2s-1,  s\right)
  \end{align}
  is equal to   (\ref{dualtwis}). 
  \end{prop}

   The archimedean functional relation for\,  $  \sum_{\pm} \  (\mathcal{F}_{\Phi}^{(\pm)} H)\left(2s-1,  s\right)$\,  has been established in \cite{Kw23} before, see  Proposition \ref{secMTcom}.  Hence,  the conclusion for  Proposition \ref{twismatchporp}  would follow from the functional equation  of $GL(3)$  (Proposition \ref{globgl3func}), as well as:

  \begin{prop}\label{agreewithCFK}
  	For any $s\in \C$, we have
  	\begin{align}
  		\mathcal{R}_{\pm}^{(a)}(s; \Phi)  \ \ =  \ \  	a^{2s-1} \cdot L(2s-1, \Phi)\cdot    \prod_{p\mid a} \ \left\{  \lambda_{\Phi}(p^{o_{p}(a)}) \ - \   \frac{\lambda_{\Phi}(p^{o_{p}(a)-1})}{p^{2(1-s)}}  \right\}. 
  	\end{align}
  \end{prop}

  This is the non-archimedean analogue of   Proposition \ref{secMTcom}. In principle, it should be a finite checking at the `ramified' places upon expanding $\mathcal{R}_{\pm}^{(a)}(s; \Phi)$  into an Euler product (as in  \cite{HY10, CIS19, BTB22+, HN22}). Unfortunately, our situation is even more complicated than the aforementioned works  ---  one has to match the right side of (\ref{agreewithCFK}) with the expression
  \begin{align}\label{badEPexp}
  	\mathcal{R}_{\pm}^{(a)}(s; \Phi) 
  	\ &= \   \sum_{der \mid a} \ \  \frac{\mu(r) \mu(d)}{d}  \left(\frac{d^{2}e  }{\widetilde{r}}\right)^{2s-1}  \overline{\lambda_{\Phi}}\left(\frac{a}{der}\right)   \ 
  	\sum_{\substack{a_{1}\ge 1 }}  \ \frac{ \lambda_{\Phi}(\widetilde{e}a_{1})}{ (a_{1})^{2s-1}} \nonumber\\
  	\ &= \ L(2s-1, \Phi)    \sum_{der \mid a} \ \  \frac{\mu(r) \mu(d)}{d}  \left(\frac{d^{2}e   }{\widetilde{r}}\right)^{2s-1}  \overline{\lambda_{\Phi}}\left(\frac{a}{der}\right)   \nonumber\\
  	&\hspace{30pt} \cdot  
  	\prod_{p \, \mid \, \widetilde{e}} \  \  \left(1-\lambda_{\Phi}(p)p^{-(2s-1)} +\overline{\lambda_{\Phi}}(p)p^{-2(2s-1)}-p^{-3(2s-1)}\right)\sum\limits_{m=0}^{\infty} \ \frac{\lambda_{\Phi}(p^{m+o_{p}(\widetilde{e})})}{(p^{m})^{2s-1}},
  \end{align}
  where  $\widetilde{e}:= e/(r,e)$ and $\widetilde{r}:= r/ (r,e)$.  Equation (\ref{badEPexp}) follows from   (\ref{restwisDS}), the  multiplicativity of the $GL(3)$ Hecke eigenvalues, and the Euler product identity
  \begin{align}\label{gl3EPid}
  	 L(s, \Phi) \ &= \  \prod_{p} \ \left(1-\lambda_{\Phi}(p)p^{-s} +\overline{\lambda_{\Phi}}(p)p^{-2s}-p^{-3s}\right)^{-1} \hspace{20pt} (\re s \ \gg \ 1)
  \end{align}
  for $L$-functions of $GL(3)$, see pp. 173-174 of \cite{Go15}.    When $a$ is prime, one can indeed show the desired agreement using (\ref{badEPexp}). However, it is already unclear how to proceed even if we assume  $a$ is square-free. Nevertheless,  we have found a simple proof that  works for any integer $a\ge 1$.

\begin{proof}[Proof of Proposition \ref{agreewithCFK}]
Suppose $\re s\gg 1$.  By Corollary \ref{anacondtwi}, we may evaluate $\mathcal{R}_{\pm}^{(a)}(s; \Phi) $ by taking the residue of (\ref{prevorofor}) at $s_{0}=2s-1$.   Switching the order of summation,  we find that the $\ell$-sum in (\ref{prevorofor}) can be identified with the \textit{Ramanujan sum} $ S(0, \mp a_{1};d)$. In other words,  
\begin{align}\label{resDS}
 \mathcal{R}_{\pm}^{(a)}(s; \Phi)  \ \ =  \ \ 	 \ \sum_{d \mid a} \  d^{4s-3} \
	\sum_{a_{1}\ge 1}  \ \frac{\mathcal{B}_{\Phi}\left( a/d, a_{1}   \right)}{ (a_{1})^{2s-1}} \cdot S(0, \mp a_{1};d),
\end{align}
 Note: the last expression is independent of the sign and we shall now drop `$\pm$' sign in our notation  $\mathcal{R}_{\pm}^{(a)}(s; \Phi)$ from now on. 

Suppose $\re w \gg 1$.  We define the triple Dirichlet series:
\begin{align}
 \mathcal{C}(w,s; \Phi) \ \ := \ \   \sum_{a=1}^{\infty} \ \frac{	\mathcal{R}^{(a)}(s; \Phi) }{a^{w}}   \ \ = \ \  \sum_{a,d, a_{1}=1}^{\infty} \  \frac{d^{4s-3}\mathcal{B}_{\Phi}(a, a_{1}) S(0,a_{1};d)}{(da)^{w}a_{1}^{2s-1}},
\end{align}
where the last expression follows from the change of variables $a\to da$. The $d$-sum can be readily computed by the well-known identity:
\begin{align}
  \sum_{d=1}^{\infty} \ \frac{S(0, a_{1};d)}{d^{w-4s+3}} \ \ = \ \   \frac{\sigma_{-w+4s-2}(a_{1})}{\zeta(w-4s+3)}.
\end{align}
  From (\ref{eisfournorm}), observe that $	\sigma_{-w+4s-2}(|a_{1}|)  =  	 |a_{1}|^{\mu} \cdot \mathcal{B}_{E^{*}(*;\,  \mu)}(a_{1})$ with $\mu:= -w/2+2s-1$.  As a result,  we recognize 
 \begin{align}
 	\mathcal{C}(w,s; \Phi) \ \ &= \ \   \frac{1}{\zeta(w-4s+3)} \ \ \sum_{a, a_{1}=1}^{\infty} \  \frac{ \mathcal{B}_{E^{*}(*;\, \mu)}(a_{1}) \mathcal{B}_{\Phi}(a, a_{1})}{(a^2 a_{1})^{w/2}} 
 \end{align}
as a double Dirichlet series of  $GL(3)\times GL(2)$  type, see (\ref{rankse}).  Next, by considering the Euler products  (\ref{3EUl})-(\ref{RSEul}) of $GL(3)$ and $GL(3)\times GL(2)$ types, we obtain 
 \begin{align}\label{genseries}
 		\mathcal{C}(w,s; \Phi)	\ = \   \frac{L(\frac{w}{2}+\mu, \Phi) L(\frac{w}{2}-\mu, \Phi)}{\zeta(w-4s+3)} 
 	\ = \  \frac{L(2s-1, \Phi)L(w+1-2s, \Phi)}{\zeta(w-4s+3)}. 
 \end{align}
Then we may  rewrite (\ref{genseries})  as a Dirichlet series in $w$ using the $GL(3)$ $L$-series  (\ref{gl3stdL}) with M\"obius inversion:
  \begin{align}
  	\mathcal{C}(w,s; \Phi)
  	\  &= \ L(2s-1, \Phi) \ \sum_{r=1}^{\infty} \ \frac{1}{r^{w}} \sum_{mn=r} \ \frac{\lambda_{\Phi}(m)\mu(n)}{m^{1-2s}n^{3-4s}}.
  \end{align}
Upon comparing coefficients,  we readily conclude that
  \begin{align}\label{conv}
  	\mathcal{R}^{(a)}(s;  \Phi) \ = \ L(2s-1, \Phi) \, \cdot\,     a^{2s-1}  \, \sum_{n\mid a} \ \frac{\lambda_{\Phi}(a/n)\mu(n)}{n^{2(1-s)}}
  \end{align}
and this holds for any $s\in \C$ by analytic continuation. The convolution of  multiplicative functions in (\ref{conv}) can be computed as
\begin{align}
 \sum_{n\mid a} \ \frac{\lambda_{\Phi}(a/n)\mu(n)}{n^{2(1-s)}}
	\ &= \       \prod_{p\mid a} \ \left\{  \lambda_{\Phi}(p^{o_{p}(a)}) \ - \   \frac{\lambda_{\Phi}(p^{o_{p}(a)-1})}{p^{2(1-s)}}  \right\}.
\end{align}
This completes the proof of Proposition \ref{agreewithCFK}.
\end{proof}

\subsection{Concluding Theorem \ref{twisdirc}}  Theorem \ref{twisdirc} follows from putting  Proposition \ref{comspec}, equation (\ref{diagoffsplt}), Proposition \ref{repaCarg}, Proposition \ref{maintwisdiag}, Proposition  \ref{twismatchporp} and equation (\ref{twistDS})  together.

 \subsection{Shape of the Dual Moment}\label{shapedualDiric}
In this section, we assume that  $a= p$ is a prime as in  \cite{Y11, BHKM20}. The connection between the twisted spectral moment  and the moment averages over Dirichlet characters is now a straight-forward consequence of the orthogonality relation thanks to the structures cast early on by our period integral construction (see Corollary \ref{incomexpf}). 

\begin{prop}\label{cuspdual}
	For any $(s_{0},s)\in \C^2$ except on the line $2s-s_{0}=1$, we have
	\begin{align}\label{expinmul}
		\mathcal{L}^{(p)}_{\pm}\left(s_{0}, s; \Phi \right)  \  \ =  \ \ 	& 	\mathcal{P}_{p}(s_{0}, s; \Phi) \,  \zeta(2s-s_{0}) \, L(s_{0}, \Phi) \nonumber\\
		& \hspace{20pt} \ + \  	\frac{ p^{ s_{0}+2s-1}}{\phi(p)}  \    \left( \ \  \sideset{}{^+}{\sum}_{\chi \ (\bmod\, p)} \ \  \mp \ \  \sideset{}{^-}{\sum}_{\chi \ (\bmod\, p)} \right) \ \tau(\chi) L(2s-s_{0}, \chi)  L\left(s_{0}, \Phi\otimes \overline{\chi}\right),
	\end{align}
	where   $\tau(\chi):= \sideset{}{^*}{\sum}\limits_{x \, (\bmod\, p)} \chi (x) e(x/p)$ is the Gauss sum of\, $\chi \, (\bmod\, p)$ and 
	\begin{align}
		\mathcal{P}_{p}(s_{0}, s; \Phi) \ \ := \ \  &\frac{p^{2s+s_{0}-1}}{\phi(p)} \, (1-p^{-(2s-s_{0})}) \left(-1 +\lambda_{\Phi}(p)p^{1-s_{0}}- \overline{\lambda_{\Phi}}(p) p^{1-2s_{0}}+p^{1-3s_{0}}\right)\nonumber\\
		 &\hspace{40pt}  \ + \ \overline{\lambda_{\Phi}}(p) \ -  \  p^{-s_{0}}.  
	\end{align}
\end{prop}

\begin{proof}
	The computations below will be first  carried out on the region of absolute convergence (\ref{regAC}). Then  our desired conclusion will follow from analytic continuation.  In  (\ref{twistDS})  we either have $d=1$ or $d=p$,  and  correspondingly   $r\mid p$ or $r\mid 1$.

	\noindent \textbf{Case (1): $d=1$. }  We take $\ell=1$ in this case and then 
	\begin{align*}
		\mathcal{L}^{(p)}_{\pm}\left(s_{0}, s; \Phi \right) \bigg|_{d=1} \ \ = \ \  	&\ \sum_{r\in \{1,p\}} \ \frac{\mu(r)}{r^{s_{0}}}\cdot   \overline{\lambda_{\Phi}}(p/r) \cdot \zeta(2s-s_{0}) L\left(s_{0}, \mp r; \Phi\right) \nonumber\\
		\ \ =  \ \ &  \left( \,  \overline{\lambda_{\Phi}}(p)- p^{-s_{0}}\right) \zeta(2s-s_{0}) L\left(s_{0}, \Phi\right).
	\end{align*}

	\noindent \textbf{Case (2): $d=p$.} We have
	\begin{align}\label{case2}
		\mathcal{L}^{(p)}_{\pm}\left(s_{0}, s; \Phi\right) \bigg|_{d=p}  \ = \ 	p^{ 2s_{0}-1}  \  \   \sideset{}{^*}{\sum}_{\ell \ (\bmod\, p)} \  \zeta\left( 2s-s_{0}, \frac{\ell}{p}\right) 	L\left( s_{0}; \mp \frac{\overline{\ell}}{p}; \Phi\right). 
	\end{align}
The Hurwitz zeta function in (\ref{case2})   can be rewritten as  
	\begin{align*}
		\zeta\left( 2s-s_{0}, \frac{\ell}{p}\right) \  \ &=  \ \   \frac{p^{2s-s_{0}}}{\phi(p)} \ \sum_{\chi \ (\bmod\, p)} \ \overline{\chi}(\ell) L(2s-s_{0}, \chi)
	\end{align*}
	 with (\ref{congHurw}) and   the orthogonality relation for Dirichlet characters.  Hence, we have
	\begin{align}\label{afterorth}
			\mathcal{L}^{(p)}_{\pm}\left(s_{0}, s; \Phi\right) \bigg|_{d=p}  \ &= \ 	 \frac{p^{2s+s_{0}-1}}{\phi(p)}  \      \sum_{\chi \ (\bmod\, p)} \  L(2s-s_{0}, \chi)  \  \sideset{}{^*}{\sum}_{\ell \ (\bmod\, p)}  \ \overline{\chi}(\ell) L\left( s_{0}; \mp \frac{\overline{\ell}}{p}; \Phi\right) \nonumber\\
			\ &=\   	 \frac{p^{2s+s_{0}-1}}{\phi(p)}  \      \sum_{\chi \ (\bmod\, p)} \  L(2s-s_{0}, \chi)  \ \sum_{n=1}^{\infty} \ \frac{\lambda_{\Phi}(n)}{n^{s_{0}}} \  \sideset{}{^*}{\sum}_{\ell \ (\bmod\, p)}  \ \overline{\chi}(\ell) 	e\left(\mp \frac{n\bar{\ell}}{p}\right).
	\end{align}
	
	If $(n,p)=1$, then the replacement\,   $\ell \to \mp n\bar{\ell}$\, $(\bmod\, p)$ implies 
	\begin{align}\label{intogauss}
		\sideset{}{^*}{\sum}_{\ell \ (\bmod\, p)}  \ \overline{\chi}(\ell) 	e\left(\mp \frac{n\bar{\ell}}{p}\right) \ \ =  \ \  \overline{\chi}(\mp n) \tau(\chi). 
	\end{align}
 (When $\chi=\chi_{0}$, equation (\ref{intogauss}) is equal to $-1$.) If  $p\mid n$, then
	\begin{align*}
		\sideset{}{^*}{\sum}_{\ell \ (\bmod\, p)}  \ \overline{\chi}(\ell) 	e\left(\mp \frac{n\bar{\ell}}{p}\right) \ \ = \  \  	\sideset{}{^*}{\sum}_{\ell \ (\bmod\, p)}  \ \overline{\chi}(\ell)  \ \ = \ \  \phi(p)\cdot  1_{\chi=\chi_{0}}. 
	\end{align*}
	
	As a result, the contribution from $\chi=\chi_{0}$ in (\ref{afterorth}) is equal to 
	\begin{align}
			\mathcal{L}^{(p)}_{\pm}\left(s_{0}, s; \Phi\right) \bigg|_{d=p, \, \chi=\chi_{0}}   \ \ &= \  \   \frac{p^{2s+s_{0}-1}}{\phi(p)} \,  L(2s-s_{0}, \chi_{0})  \left\{ -\sum_{(n,p)=1} \ \frac{\lambda_{\Phi}(n)}{n^{s_{0}}}  \ + \ \phi(p) \, \sum_{p\mid n} \ \frac{\lambda_{\Phi}(n)}{n^{s_{0}}}   \right\} \nonumber\\
		\ \ &= \ \     \frac{p^{2s+s_{0}-1}}{\phi(p)} \,  L(2s-s_{0}, \chi_{0})  \left\{  \phi(p) L(s_{0}, \Phi) \ - \  p \sum_{(n,p)=1} \ \frac{\lambda_{\Phi}(n)}{n^{s_{0}}} \right\}. \nonumber
	\end{align}
	It follows from (\ref{gl3EPid}) that
	\begin{align}
			\mathcal{L}^{(p)}_{\pm}\left(s_{0}, s; \Phi\right) \bigg|_{d=p, \, \chi=\chi_{0}} 	\ \ &= \  \     \frac{p^{2s+s_{0}-1}}{\phi(p)} \,  L(2s-s_{0}, \chi_{0})  \, L(s_{0}, \Phi) \nonumber\\
		& \hspace{60pt} \,  \cdot \,  \left\{  \phi(p) \ - \  p   \left(1-\lambda_{\Phi}(p)p^{-s_{0}} +\overline{\lambda_{\Phi}}(p)p^{-2s_{0}}-p^{-3s_{0}}\right) \right\}  \nonumber\\
		\ \ &= \  \  \frac{p^{2s+s_{0}-1}}{\phi(p)} \, (1-p^{-(2s-s_{0})}) \left(-1 +\lambda_{\Phi}(p)p^{1-s_{0}}- \overline{\lambda_{\Phi}}(p) p^{1-2s_{0}}+p^{1-3s_{0}}\right) \nonumber\\
		&\hspace{60pt} \cdot \,  \zeta(2s-s_{0}) \, L(s_{0}, \Phi). \nonumber
	\end{align}
	
	On the other hand,  the contributions from $\chi\neq \chi_{0}$ in (\ref{afterorth}) is equal to 
	\begin{align}
			\mathcal{L}^{(p)}_{\pm}\left(s_{0}, s; \Phi\right) \bigg|_{d=p, \, \chi\neq \chi_{0}}  \ \ &= \  \   \frac{p^{2s+s_{0}-1}}{\phi(p)}  \     \sideset{}{^*}{\sum}_{\chi \, (\bmod\, p)} \  \chi(\mp 1)\tau(\chi) L(2s-s_{0}, \chi)  \, \sum_{(n,p)=1} \ \frac{\lambda_{\Phi}(n)\overline{\chi}(n)}{n^{s_{0}}}   \nonumber\\
			\ \ &= \ \   \frac{p^{2s+s_{0}-1}}{\phi(p)}  \     \sideset{}{^*}{\sum}_{\chi \, (\bmod\, p)} \  \chi(\mp 1)\tau(\chi) L(2s-s_{0}, \chi) L(s_{0}, \Phi\otimes \overline{\chi}). \nonumber
	\end{align}
	The conclusion of Proposition \ref{cuspdual} follows readily from
	\begin{align*}
			\mathcal{L}^{(p)}_{\pm}\left(s_{0}, s; \Phi \right)  \ \ = \ \  	\mathcal{L}^{(p)}_{\pm}\left(s_{0}, s; \Phi \right) \bigg|_{d=1} \ + \  	\mathcal{L}^{(p)}_{\pm}\left(s_{0}, s; \Phi\right) \bigg|_{d=p, \, \chi=\chi_{0}}   \ + \ 	\mathcal{L}^{(p)}_{\pm}\left(s_{0}, s; \Phi\right) \bigg|_{d=p, \, \chi\neq \chi_{0}}  
	\end{align*}
	and   analytic continuation. 

\end{proof}

It is certainly possible, but combinatorially involved, to obtain a version of (\ref{expinmul}) when $a$ is composite, i.e.,   rewrite $\mathcal{L}^{(a)}_{\pm}\left(s_{0}, s; \Phi \right)$ in terms of automorphic $L$-functions twisted by \textit{primitive} Dirichlet characters.  If $\Phi \cong 1 \boxplus 1 \boxplus 1$ or $\phi \boxplus 1$ with $\phi$ being a cusp form of $GL(2)$, the key technicalities behind seem to have been worked out in the recent work \cite{DHKL20}.   We will not pursue this task in order to maintain a reasonable length for this article.


\section{Concluding Remarks}\label{concrem}

\subsection{Relation to the 4th Moment of Dirichlet $L$-functions}\label{rel4thDIr}
We specialize to  $s=1/2$,\, $\re s_{0}=1/2$\,  and $\Phi= (E_{\min}^{(3)})^{^*}( \ * \ ; (0,0,0))$ in  (\ref{expinmul}). Once again, we  restrict ourselves  to  prime conductors only.  In this case, we have
\begin{align}
	\mathcal{L}^{(p)}_{\pm}\left(s_{0}, 1/2; \Phi \right)  \  \ =  \ \ 	& 	\mathcal{P}_{p}(s_{0}, s) \ \zeta(1-s_{0}) \zeta(s_{0})^3 \nonumber\\
	& \hspace{20pt} \ + \  	\frac{ p^{ s_{0}}}{\phi(p)}  \    \left( \ \  \sideset{}{^+}{\sum}_{\chi \ (\bmod\, p)} \ \  \mp \ \  \sideset{}{^-}{\sum}_{\chi \ (\bmod\, p)} \right) \ \tau(\chi) L(1-s_{0}, \chi)  L\left(s_{0}, \overline{\chi}\right)^3. 
\end{align}

Recall the functional equation for the primitive Dirichlet $L$-functions:
\begin{align}
	L(s_{0}, \overline{\chi}) \ \ = \ \  i^{-a_{\chi}}\, \frac{\tau(\overline{\chi})}{\sqrt{p}} \ p^{\frac{1}{2}-s_{0}} \  \frac{\Gamma_{\R}(1-s_{0}+a_{\chi})}{\Gamma_{\R}(s_{0}+a_{\chi})} \  L(1-s_{0}, \chi),
\end{align}
where $a_{\chi}=0$ if $\chi(-1)=1$ and $a_{\chi}=1$ if $\chi(-1)=-1$.  Using also the fact that $|\tau(\chi)|=\sqrt{p}$, it follows that
\begin{align}
	\sideset{}{^+}{\sum}_{\chi \ (\bmod\, p)}  \    \tau(\chi) L(1-s_{0}, \chi)  L\left(s_{0}, \overline{\chi}\right)^3 \  \ &= \ \  p^{1-s_{0}} \  \frac{\Gamma_{\R}(1-s_{0})}{\Gamma_{\R}(s_{0})}  \  \ 	\sideset{}{^+}{\sum}_{\chi \ (\bmod\, p)}  \  L(1-s_{0}, \chi)^2  L\left(s_{0}, \overline{\chi}\right)^2, \nonumber \\
	\sideset{}{^-}{\sum}_{\chi \ (\bmod\, p)}  \    \tau(\chi) L(1-s_{0}, \chi)  L\left(s_{0}, \overline{\chi}\right)^3 \  \ &= \ \  -i p^{1-s_{0}} \  \frac{\Gamma_{\R}(2-s_{0})}{\Gamma_{\R}(1+s_{0})}  \ \   \sideset{}{^-}{\sum}_{\chi \ (\bmod\, p)}  \  L(1-s_{0}, \chi)^2  L\left(s_{0}, \overline{\chi}\right)^2. 
\end{align}

The dual moment of 
\begin{align}
	& 2\  \sum_{j=1}^{\infty} \ H(\mu_{j}) \  \frac{ \lambda_{j}(a)\Lambda(1/2, \phi_{j})^3}{\langle \phi_{j}, \phi_{j}\rangle}  + \frac{1}{2\pi } \, \int_{\R} \ H(i\mu) \ \frac{ \sigma_{-2\mu}(a) a^{-\mu} |\Lambda(1/2+i\mu)|^6 }{|\Lambda(1+2i\mu)|^2} \ d\mu
\end{align}
is given by 
\begin{align}
p^{-1/2} \	\sum_{\pm} \ \  \int_{(1/2)} \  	\mathcal{L}^{(p)}_{\pm}\left(s_{0}, s; \Phi \right)   (\mathcal{F}_{\Phi}^{(\pm)}H)\left(s_{0},  s\right) \ \frac{ds_{0}}{2\pi i}, 
\end{align}
which is readily observed  to be a  sum of the following three weighted moments of $GL(1)$ $L$-functions:
\begin{align}\label{eisdualmo}
	&\hspace{60pt} p^{-1/2 } \, \int_{(1/2)} \ \mathcal{P}_{p}(s_{0}, s) \ |\zeta(s_{0})|^{4}  \   \frac{\Gamma_{\R}(1-s_{0})}{\Gamma_{\R}(s_{0})}  \cdot \sum_{\pm} \  (\mathcal{F}_{\Phi}^{(\pm)}H)\left(s_{0},  s\right)  \ \frac{ds_{0}}{2\pi i}, \nonumber\\[0.1in]
&\hspace{40pt}  \frac{p^{1/2}}{\phi(p)} \ \ \sideset{}{^+}{\sum}_{\chi \ (\bmod\, p)} \ \int_{(1/2)} \ \left|L(s_{0}, \chi)\right|^4 \ \frac{\Gamma_{\R}(1-s_{0})}{\Gamma_{\R}(s_{0})}   \cdot   \sum_{\pm} \  (\mathcal{F}_{\Phi}^{(\pm)}H)\left(s_{0},  s\right)  \ \frac{ds_{0}}{2\pi i},  \nonumber\\[0.1in]
	  \  &\frac{p^{1/2}}{\phi(p)} \ \ \sideset{}{^-}{\sum}_{\chi \ (\bmod\, p)} \ \int_{(1/2)} \ \left|L(s_{0}, \chi)\right|^4 \, \frac{\Gamma_{\R}(2-s_{0})}{\Gamma_{\R}(1+s_{0})} \cdot   i \left\{ (\mathcal{F}_{\Phi}^{(+)}H)\left(s_{0},  s\right)- (\mathcal{F}_{\Phi}^{(-)}H)\left(s_{0},  s\right)\right\} \ \frac{ds_{0}}{2\pi i}. \nonumber\\ 
\end{align}

\begin{rem}\label{reas}
	While the twists appeared on each side of the reciprocity formula (\ref{basicmototwis}),  namely $\left\{\lambda_{f}(p)\right\}$ and  $\left\{\chi  \bmod\, p\right\}$,  are seemingly unrelated as discussed in Section \ref{reciprsec},  the additive harmonics
	\begin{align}\label{addharper}
		\mathcal{E}_{p} \ := \  \left\{ \, e(\mp a_{1}\bar{\ell}/p): \, a_{1}\ge 1,\  \ell \ (\bmod\, p),\  (\ell, p)=1\, \right\}
	\end{align}
	obtained in our approach  connects the mentioned multiplicative twistings. Indeed,  $\mathcal{E}_{p}$ naturally relates to $\left\{\chi  \bmod\, p\right\}$ by the orthogonality relation (see Proposition \ref{cuspdual}), whereas  $\mathcal{E}_{p}$ pertains to   $\left\{\lambda_{f}(p)\right\}$  via  our period integral construction, more precisely,   the arithmetic factor  originating from the \textit{unipotent} translate of the \textit{Whittaker function} (see Corollary \ref{incomexpf}). 
	
	Additive harmonics of various shapes are involved in the methods of \cite{Y11, BHKM20}  as well,  but certainly in rather different ways, e.g.,  Kloosterman sums, Kuznetsov/ Poisson/ Voronoi formulae, $\delta$-symbol expansion etc. 
\end{rem}


	\subsection{The Twisted Cubic Moment Conjecture}
	
	In \cite{Kw23a+}, we stated and proved the untwisted version of the following conjecture:
	
		\begin{conj}\label{twiscubconj}
			For any integer $a\ge 1$,  the full set of the main terms for 
			\vspace{0.1in}
			\begin{align}\label{twishifcub}
				& \sum_{j=1}^{\infty} \ H(\mu_{j}) \  \frac{ \lambda_{j}(a)\Lambda(1/2- \alpha_{1}, \phi_{j})\Lambda(1/2- \alpha_{2}, \phi_{j})\Lambda(1/2- \alpha_{3}, \phi_{j})}{\langle \phi_{j}, \phi_{j}\rangle} \nonumber\\
				&\hspace{120pt} + \frac{1}{4\pi } \ \int_{\R} \ H(i\mu) \ \frac{ \sigma_{-2\mu}(a) a^{-\mu}\prod\limits_{i=1}^{3}\Lambda(1/2+i\mu -\alpha_{i}) \Lambda(1/2+i\mu+ \alpha_{i})}{|\Lambda(1+2i\mu)|^2} \ d\mu
			\end{align}
			is given by
			\begin{align}\label{twicubmo}
				&\hspace{20pt}     \frac{ 1}{2} \cdot \sum_{\epsilon_{1}, \epsilon_{2}, \epsilon_{3}= \pm 1} \ \zeta(1 - \epsilon_{1}\alpha_{1}- \epsilon_{2}\alpha_{2})  \zeta(1 - \epsilon_{1}\alpha_{1}-\epsilon_{3}\alpha_{3})  \zeta(1 - \epsilon_{2}\alpha_{2}- \epsilon_{3}\alpha_{3})  \nonumber\\
				&\hspace{60pt} \ \cdot \	a^{-\frac{1}{2}}  \ \prod_{p\mid a} \ \left\{   \tau_{-\epsilon\cdot \alpha}\left(p^{o_{p}(a)}\right) \ - \   \frac{\tau_{-\epsilon\cdot\alpha}\left(p^{o_{p}(a)-1}\right) }{p} \right\} \nonumber\\
				&\hspace{140pt}  \cdot 	 \int_{(0)}  \frac{H(\mu)}{|\Gamma(\mu)|^2} \cdot \prod_{i=1}^{3} \ \prod_{\pm}  \ \Gamma_{\R}\left(1/2 \pm \mu -\epsilon_{i}\alpha_{i}\right) \ \frac{d\mu}{2\pi i},
			\end{align}
			where $\alpha_{1}+\alpha_{2}+\alpha_{3}=0$, $\epsilon \cdot \alpha := (\epsilon_{1}\alpha_{1}, \epsilon_{2}\alpha_{2}, \epsilon_{3}\alpha_{3})$,   $\tau_{\alpha}(n):= \sum\limits_{d_1d_2d_3 = n} d_1^{-\alpha_1} d_2^{-\alpha_2} d_3^{-\alpha_3} $, and $(\phi_{j})_{j=1}^{\infty}$ is an orthogonal basis of  even  Hecke-normalized  Maass cusp forms of $SL_{2}(\Z)$ which satisfy $\Delta\phi_{j}= ( 1/4-\mu_{j}^2) \, \phi_{j}$. Note: $(\epsilon_{1}, \epsilon_{2}, \epsilon_{3})= (+1, +1, +1)$ corresponds to the $0$-swap term. 
		\end{conj}

In this article and \cite{Kw23a+}, we have established all of the archimedean functional relations  for showing agreement with the prediction (\ref{twicubmo}),  as well as the non-archimedean counterparts for the $0$- and $3$-swap terms.  The identities for the $1,2$-swap terms should follow in a similar fashion using the method of this article. (In \cite{Kw23a+}, this is very much true for the archimedean case.)  As indicated  by (\ref{eisdualmo}) and the main theorem of \cite{Kw23a+},  the full set of  main terms for the fourth moment of the Dirichlet $L$-functions  \`{a} la CFKRS should also be visible in the full spectral identity for (\ref{twishifcub}). This offers an alternative approach to  assemble the main terms,  distinct from the one of  \cite{Y11}. We shall leave these to a future work.  

\subsection{Supplementary Details Regarding the Integral Transform}\label{suppInt}

Recall the double integral of equation (\ref{cleanup}):
\begin{align}\label{doukeyint}
\hspace{20pt} 	 \int_{0}^{\infty} \int_{0}^{\infty} h(y_{1})\,  (y_{0}^2 y_{1})^{s-\frac{1}{2}}  \, e\left( \frac{a_{1}}{da_{0}}\, \frac{y_{0}^2}{1+y_{0}^2}\right)   W_{-\alpha(\Phi)} \left(\left|\frac{a_{1}}{da_{0}}\right|\, \frac{ y_{0}}{ 1+y_{0}^2}, \ y_{1}\sqrt{1+y_{0}^2}   \right)   \  \frac{dy_{0} \, dy_{1}}{y_{0}y_{1}^2}.
\end{align}
The Mellin inversion formula
\begin{align}
	f(y) \ = \ \int_{(\sigma_{0})} \, \left(\int_{0}^{\infty}\, f(X)X^{s_{0}-1} \, d^{\times}X\right) y^{1-s_{0}} \  \frac{ds_{0}}{2\pi i}
\end{align}
together with rearrangement of integrals imply  (\ref{doukeyint}) is equal to 
\begin{align}
\int_{(\sigma_{0})} \,  \left|\frac{a_{1}}{da_{0}}\right|^{1-s_{0}} \,  \int_{0}^{\infty} \int_{0}^{\infty} \, h(y_{1})\,  &(y_{0}^2 y_{1})^{s-\frac{1}{2}} 	\nonumber\\
&\hspace{-50pt} \cdot \int_{0}^{\infty} \, W_{-\alpha(\Phi)} \left(X \frac{ y_{0}}{ 1+y_{0}^2}, \ y_{1}\sqrt{1+y_{0}^2}   \right)  e\left( \pm X \frac{y_{0}^2}{1+y_{0}^2}\right)  X^{s_{0}-1} \, d^{\times} X\, \frac{dy_{0}\, dy_{1}}{y_{0}y_{1}^2} \frac{ds_{0}}{2\pi i}. 
\end{align}
Upon making the changes of variables $X\to X((1+y_{0}^2)/ y_{0}) $ and $y_{1}\to y_{1}/ \sqrt{1+y_{0}^2}$, one can further rewrite the above as
\begin{align}
	\int_{(\sigma_{0})} \,  \left|\frac{a_{1}}{da_{0}}\right|^{1-s_{0}} \,  \int_{0}^{\infty} \int_{0}^{\infty} \,  h\left(\frac{y_{1}}{\sqrt{1+y_{0}^2}}\right)  &  \int_{0}^{\infty} \, W_{-\alpha(\Phi)} \left(X , y_{1}   \right)  e\left( \pm Xy_{0}\right)  X^{s_{0}-1} \, d^{\times} X\ \nonumber\\
&\hspace{60pt}  \cdot  \frac{y_{0}^{2s-s_{0}}}{(1+y_{0}^2)^{\frac{s}{2}+\frac{1}{4}-s_{0}}} \, 	y_{1}^{s-\frac{1}{2}} \  \frac{dy_{0}\, dy_{1}}{y_{0}y_{1}^2}  \frac{ds_{0}}{2\pi i}. 
\end{align}
Plugging this into (\ref{cleanup}), the display (\ref{prop72kw23}) with (\ref{duaintrafir}) readily follow.  

In a similar manner, the argument appeared in  \cite[Section 4]{Kw23} leads to 
\begin{align}\label{gener}
(\ref{doukeyint}) \ =  \  \int_{(\sigma_{0})} \,  \left|\frac{a_{1}}{da_{0}}\right|^{1-s_{0}} \,   \int_{0}^{\infty} \int_{0}^{\infty} \, & h(y_{1}^{-1}) (y_{0}^2 y_{1}^{-1})^{s-\frac{1}{2}}  \nonumber\\
&\hspace{-70pt} \cdot \int_{0}^{\infty} \, 	W_{\rho(w_{\ell}) \Phi} \left[\begin{pmatrix}
	\pm X & & \\
	1 & 1& \\
	   &  & 1
\end{pmatrix} \begin{pmatrix}
y_{0}y_{1} & &  \\
                  & y_{1} & \\
                  &         & 1
\end{pmatrix}\right] X^{s_{0}-1} \, d^{\times} X \, d^{\times} y_{0} \, dy_{1}\, \frac{ds_{0}}{2\pi i}. 
\end{align}
 This yields  a slightly more intrinsic form of the integral transform, for (\ref{gener}) also holds for non-spherical test vector $\Phi$  and  it does not appeal to any matrix decomposition (say Iwasawa or Bruhat).  Of course, when $\Phi$ is specialized to be the spherical vector,  (\ref{gener})  is technically equivalent to the corresponding calculations  in this article as well as \cite{Kw23, Kw23a+}.

 \section{Acknowledgement}
The research is supported by the EPSRC grant: EP/W009838/1.   Part of the work was completed during the author's stay at  American Institute of Mathematics, Texas A\&M University,  the Chinese University of Hong Kong  and Queen's University.    The author would like to express his gratitude for  their generous hospitality.  It is a great pleasure to thank the reviewer for a careful reading of the article and his/ her valuable comments. 

\section{Data availability}
No data was generated or analyzed as part of the writing of this paper.

\section{Declarations}
On behalf of all authors, the corresponding author states that there is no conflict of interest.

\ \\
\end{document}